\documentclass[11pt]{amsart}
\usepackage{amsmath,amssymb,amsfonts,amsthm,amscd,indentfirst}
\newdimen\AAdi%
\newbox\AAbo%
\def\AAk#1#2{\s_etbox\AAbo=\hbox{#2}\AAdi=\wd\AAbo\kern#1\AAdi{}}%
\def\AAr#1#2#3{\s_etbox\AAbo=\hbox{#2}\AAdi=\ht\AAbo\raise#1\AAdi\hbox{#3}}%
\font\tenmsb=msbm10 at 11pt \font\sevenmsb=msbm7 at 8pt
\font\fivemsb=msbm5 at 6pt
\newfam\msbfam
\textfont\msbfam=\tenmsb \scriptfont\msbfam=\sevenmsb
\scriptscriptfont\msbfam=\fivemsb

\renewcommand{\theequation}{\thesection\arabic{equation}}

\font\bbbld=msbm10 scaled\magstephalf

\newcommand{\ba}{\begin{array}}
\newcommand{\ea}{\end{array}}

\newcommand{\bfR}{\hbox{\bbbld R}}
\newcommand{\bfS}{\hbox{\bbbld S}}

\def\aint{\frac{\ \ }{\ \ }{\hskip -0.4cm}\int}

\newtheorem{theorem}{Theorem}[section]
\newtheorem{lemma}[theorem]{Lemma}
\newtheorem{proposition}[theorem]{Proposition}
\newtheorem{corollary}[theorem]{Corollary}

\theoremstyle{definition}

\theoremstyle{remark}
\newtheorem{remark}[theorem]{Remark}

\begin{document}
\setlength{\baselineskip}{1.2\baselineskip}

\title[Gauss curvature flow]{\bf Entropy and a convergence theorem for Gauss curvature flow in high dimension}

\author{Pengfei Guan}
\address{Department of Mathematics and Statistics\\
        McGill University\\
        Montreal, Quebec, H3A 2K6, Canada.}
\email{guan@math.mcgill.ca}
\author{Lei Ni}
\address{Department of Mathematics, University of California at San Diego, La Jolla, CA 92093, USA}
\email{lni@math.ucsd.edu}
\thanks{The research of the first author is partially supported by an NSERC Discovery Grant, the research of the second author is partially supported by NSF grant DMS-1105549.}
\begin{abstract} In this paper we prove uniform regularity estimates for the normalized Gauss curvature flow in higher dimensions.  The convergence of solutions in $C^\infty$-topology to a smooth strictly convex soliton as $t$ approaches to infinity is obtained as a consequence of these estimates together with an earlier result of Andrews. The estimates are established via the study of a new entropy functional for the flow.
\end{abstract}

\maketitle

\leftskip 0 true cm \rightskip 0 true cm
\renewcommand{\theequation}{\arabic{equation}}
\setcounter{equation}{0} \numberwithin{equation}{section}

\section{Introduction}

The Gauss curvature flow was introduced by Firey \cite{Firey} to model the changing shape of a tumbling stone subjected to collisions from all directions with uniform frequency.
Suppose that $\{M_t\}\subset \mathbb R^{n+1}$ is a family of compact smooth strictly convex hypersurfaces with $t\in [0,T)$.
Denote by $X(x,t)$ and $K(x, t)$ the position vector and the Gauss curvature of $M_t$.
$\{M_t\}$ is a solution of the Gauss curvature flow, if $X(x, t)$ satisfies the equation
\begin{equation}\label{gauss-1}
\frac{\partial   X(x, t)}{\partial t}=-K(x, t)\nu(x,t),
\end{equation} where $\nu(x,t)$ is the unit exterior normal of the hypersurface $M_t$.

Assuming the existence, uniqueness and regularity of the solution, Firey proved that if the initial convex surface ($M_0\subset \mathbb{R}^3$) is symmetric with respect to the origin (also called centrally symmetric), then the flow (\ref{gauss-1}) contracts the initial surface into a point in finite time and becomes spherical in shape in the process. The last statement can be rephrased that the normalized  flow (with preserved enclosed volume) converges to a round sphere. He conjectured that the result holds in general. After this initial work, the existence and uniqueness of the Gauss curvature flow in any $\mathbb{R}^{n+1}$ was established by Chou \cite{Tso}. In the same paper it was also proved that the Gauss curvature flow contracts the initial convex hypersurface into a point in finite time. More than a decade later, in a breakthrough work \cite{An3}, Andrews proved that the normalized flow in $\mathbb{R}^3$ does converge to a round sphere, namely evolving surfaces become spherical in the process, hence proving the conjecture of Firey. The proof of Andrews \cite{An3} relies on a pinching estimate, which makes use that the surface is $2$-dimensional in a crucial way. It then remains an interesting question whether or not the same picture holds in high dimensions.

In this paper, we establish uniform regularity of solutions to the normalized Gauss curvature flow. By Chou's work, the convex hypersurfaces $M_t$ (and the enclosed convex body $\Omega_t$) shrink to a point along the Gauss curvature flow at a finite time $T$. If we choose this limiting point as the origin and normalize $M_t$ such that the  enclosed volume (namely the Lebesgue measure $|\Omega_t|$) is equal to the volume of the unit ball, the normalized Gauss flow satisfies the equation:
\begin{equation}\label{gauss-nor1}
\frac{\partial   X(x, t)}{\partial t}=(-K(x, t)+u(x,t))\nu(x,t),
\end{equation}
where $u(x,t)=\langle X(x,t),\nu(x,t)\rangle $ is the supporting function.

The following is the main result of the paper.

\begin{theorem}\label{mainT}
Suppose that $M_0$ is a compact strictly convex hypersurface in $\mathbb R^{n+1}$ such that volume of the enclosed convex body  is equal to the volume of the unit ball $B_1(0)$. Assume the origin is the contracting point of the un-normalized flow (\ref{gauss-1}).  Let $\{\Omega_t\}$ be the convex bodies enclosed by $\{M_t\}$, the solution to the normalized flow (\ref{gauss-nor1}) with the above normalization. Then there exists positive constant $\Lambda\ge 1$ depending only on $n$ and $M_0$ such that,
\begin{equation} B_{\frac{1}{\Lambda}}(0)\subset \Omega_t \subset B_{\Lambda}(0), \quad \forall\quad  0\le t<\infty.\end{equation}
Moreover, for any integer $k\ge 1$, there is a constant $C(n,k,M_0)$ depending on on $n, k$ and the initial hypersurface $M_0$, such that,
\begin{equation}\label{est}
\|M_t\|_{C^k}\le C(n,k, M_0).\end{equation}
Finally, the flow (\ref{gauss-nor1}) converges in $C^\infty$-topology to a smooth strictly convex soliton $M_{\infty}$ satisfying equation
\begin{equation}\label{soliton-e}
K(x)=u(x), \forall x\in M_{\infty}.
\end{equation}
\end{theorem}

 Related to the above result,  in \cite{Hamilton-gauss}, R. Hamilton obtained the upper bound of the diameter and upper bound for the Gauss curvature for the normalized flow. In view of the Blaschke selection theorem and a general $C^\infty$-convergence result of Andrews \cite{Andrews-IMRN} which assumes the regularity of the limiting soliton, the contribution of this paper  is mainly  on the uniform $C^2$-estimates for the normalized Gauss curvature flow.  This $C^2$-estimate relies  on a $C^0$-estimate on the support function $u(x, t)$ (particularly a uniform lower bound) and a uniform lower estimate on the Gauss curvature. To prove that the support function $u(x, t)$ of solution to (\ref{gauss-nor1}) is uniformly bounded from below by a positive constant, we need to introduce a new entropy functional $\mathcal{E}(\Omega_t)$ (see Section 2 for the definition) for the enclosed convex body $\Omega_t$.   The nonnegativity of the entropy is deduced from the classical  Blachke-Santal\'o inequality \cite{Sch}. The monotonicity of the entropy along the flow, as well as geometric estimates in terms of the entropy, plays the basic role.  As a by-product of our study of this new entropy functional we deduce the non-negativity of Chow's entropy \cite{Chow},  as well as the nonnegativity of Firey's entropy \cite{Firey} (which is only defined with respect to the above normalization by placing the limiting point at the origin) for the non-centrally symmetric case. Both cases are not known previously despite the use of the terminology. Above mentioned upper bounds of Hamilton on the diameter and the Gauss curvature can also be derived from the uniform lower bound on $u(x,t)$ proved here.

It remains an open question whether or not the sphere is the only compact soliton with positive Gauss curvature. On this we prove that the unit sphere is stable among the admissible variations. We also show that for the solitons with the normalized enclosed volume, there exists a sharp lower estimates on the volume of the dual body, which implies Firey's uniqueness among solitons with central symmetry. The interested reader should consult \cite{Andrews-IMRN, Chow,  Das-Ham, Das-Lee, Hamilton-gauss} for earlier works and further references on the subject.

\section{An entropy functional and basic properties}

Let $ \Omega$ be a bounded closed convex body such that $0\in \Omega\subset \bfR^{n+1}$ and $M\doteqdot \partial \Omega$. Let $u: \bfS^n \to \bfR$ be the support function of $\Omega$,  which is defined for any $x\in \bfS^n$ by
$$
u(x)\doteqdot \max_{z\in \Omega} \langle x, z\rangle=\max_{z\in M}\langle x, z\rangle.
$$
Generally for any $z_0\in \Omega$, one can define the support function with respect to $z_0$ as
$$
u_{z_0}(x)\doteqdot \max_{z\in \Omega} \langle x, z-z_0\rangle.
$$
Define an entropy functional $\mathcal{E}(\Omega)$ by:
\begin{eqnarray*} \mathcal{E}(\Omega)\doteqdot\sup_{u_{z_0}>0}\frac1{\omega_n}\int_{\mathbb S^n} \log u_{z_0} (x) \, d\theta(x).\end{eqnarray*}
Here $\omega_n$ is the area of $\bfS^n$, $d\theta$ is the induced surface measure and the supremum  is taken among all  positive support functions.
(Later we shall show that given a non-degenerate, namely of full dimensional, convex body the entropy in fact can be attained by a positive support function.) It is easy to see that $\mathcal{E}(\Omega)$ is finite.
 In fact, since $u_{z_0}\le \operatorname{diam}(\Omega)$,  the diameter of $\Omega$ and $\mathcal{E}(\Omega)\le \log \operatorname{diam}(\Omega)$. Also denote by  $B(1)$, the unit ball $\mathbb{B}^{n+1}(1)\subset \mathbb{R}^{n+1}$. This quantity was introduced by Firey \cite{Firey} for symmetric convex bodies, there supremum is not necessary. Related quantities were also considered in \cite{Chou-wang}.

Since the non-negativity is the defining property of the entropy concept in physics \cite{evans}, the following result, as well as later monotonicity of $\mathcal{E}(\Omega)$ under the Gauss curvature flow, partially justifies the use of the terminology.

\begin{proposition}\label{key} Suppose $\Omega$ is a bounded convex body in $\mathbb R^{n+1}$ with $V(\Omega)=V(B(1))$ (here $V(\Omega)$ denotes the volume of $\Omega$). Let $z_s\in \Omega$ be the Santal\'o point of $\Omega$. Let $u_s$ be the support function with respect to  $z_s$. Then the estimate
\begin{eqnarray} &&\frac1{\omega_n}\int_{\mathbb S^n} \log u_s\ge 0,\label{est-basic1}\end{eqnarray}
holds with the equality if and only if  $\Omega$ is a round ball centered at $z_s$.
In particular
 $\mathcal{E}(\Omega)\ge 0$, and the inequality is strict unless $\Omega$ is a round ball centered at $z_s$. Moreover, for general convex body $\Omega$ (without volume normalization), we have
\begin{equation}\label{est-basic11}
\mathcal{E}(\Omega)\ge \frac{\log V(\Omega)-\log V(B(1))}{n+1}.\end{equation}\end{proposition}

Before the proof, we recall  the definition of the Santal\'o  point of $\Omega$. First given $\Omega$ and any $z_0\in \operatorname{Int}(\Omega)$ define $\Omega_{z_0}$ the polar dual of $\Omega$ with respect to $z_0$ by $\{ y+z_0\, |\, \max_{z\in \Omega}\langle y, z-z_0\rangle \le 1\}$. The Santal\'o point   is the unique point $z_s$ such that the associated polar dual $\Omega^*_{z_s}$,  has the minimum volume among all possible polar dual with respect to all possible $z_0\in \Omega$ (in fact it suffices to consider $z_0\in \operatorname{Int}(\Omega)$, the interior of $\Omega$). When $z_s$ is the Santal\'o point we also denote $\Omega^*_{z_s}$ by $\Omega^*_s$ and denote by $\Omega_s$ the translation of $\Omega$ by $-z_s$.

\begin{proof} Let $\Omega^*_s$ be the polar dual of $\Omega$ with respect to $z_s$, the Santal\'o point, its volume can be computed \cite{MP} as
\begin{eqnarray*} V(\Omega_s^*)=\frac{1}{n+1}\int_{\mathbb S^n} \frac1{u^{n+1}_s}\, d\theta.\end{eqnarray*}
Jensen's inequality yields,
\begin{eqnarray*} V(\Omega^*_s) &\ge& \frac{\omega_n}{n+1}\exp\left({\frac1{\omega_n}\int_{\mathbb S^n}\log\left(\frac1{u^{n+1}_{s}}\right)d\theta}\right)\\
& =&V(B(1))\exp\left({-\frac{n+1}{\omega_n}\int_{\mathbb S^n}\log u_s}\right)\end{eqnarray*}
Since $V(\Omega)=V(B(1))$, together with the Blaschke-Santal\'o inequality
\begin{eqnarray} V(\Omega)V(\Omega_s^*)\le V(B(1))^2 \label{BS}\end{eqnarray}
we have that
\begin{eqnarray*} V(B(1))\exp\left(-\frac{n+1}{\omega_n}\int_{\mathbb S^n}\log u_s\right)\le V(\Omega_s^*)\le \frac{V(B(1))^2}{V(\Omega)} \end{eqnarray*}
from which it is easy to see (\ref{est-basic1}). The estimate (\ref{est-basic11}) follows similarly.
If the equality holds,  the Jensen's inequality in the first step of the proof is an equality.
Since $e^x$ is  strictly convex, $\frac1{u^{n+1}_{s}}$, hence $u_s$ is constant. It must be $1$ as $V(\Omega^*_s)=V(B(1))$. Hence $\Omega$ is a ball centered at $x_s$.

As for the final statement of the proposition, we note that $\mathcal{E}(\Omega)\ge \frac1{\omega_n}\int_{\mathbb S^n} \log u_s\ge 0$. By  lemma below which asserts that the entropy is attained at a unique point $z_e$, we deduce that  the assumption $\mathcal{E}(\Omega)= 0$ implies that $\frac1{\omega_n}\int_{\mathbb S^n} \log u_s= 0$ and that $x_s$  is the point where entropy is achieved, namely $x_s=z_e$. By the above proof we have that  $\Omega$ is a ball of radius $1$ centered at $x_s$.  \end{proof}
\medskip

 A geometric approach to the previous result is as follows. For any point $z_0\in \Omega$, by the definition, the dual body $\Omega^*_{z_0}$ is defined by the equation
$$
\Omega^*_{z_0} -z_0=\{ w\, |\, \langle w, z-z_0\rangle \le 1, \forall z\in \Omega\}.
$$
Write $w$ in terms of polar coordinates we have that
\begin{equation}\label{def-dual-body1}
\Omega^*_{z_0} -z_0=\{(r, x)\, |\,  r u_{z_0}(x)\le 1\}.
\end{equation}
Here $u_{z_0}(x)$ is the support function of $\Omega$ with respect to $z_0$. This in particular implies that
$$
V(\Omega^*_{z_0})=\int_0^{1/u_{z_0}(x)} \int_{\mathbb{S}^n}r^n \, d\theta\, dr=\frac{1}{n+1}\int_{\mathbb{S}^n}\frac{1}{u_{z_0}^{n+1}}\, d\theta.
$$

If we normalize the volume of $\Omega$ to be that of the unit ball, Blaschke-Santal\'o inequality implies that
there exists $z_0\in \Omega$ such that $|\Omega^*_{z_0}|\le V(B(1))$. In the case that $\Omega$ is not affine equivalent to the unit ball, such $z_0$ forms an open subset.
Now observe the following
geometric interpretation of the quantity $\int_{\mathbb{S}^n}\log u_{z_0}(x)$.

\begin{proposition}  Let $\Omega^{0}_{z_0}=\Omega^*_{z_0}-z_0$.
Then
$$
\int_{\mathbb{S}^n} \log u_{z_0}(x)\, d\theta(x)=\left(\int_{B(1)\setminus\Omega^{0}_{z_0}}- \int_{\Omega^{0}_{z_0}\setminus B(1)}\right) \frac{1}{|w|^{n+1}}\, d\mu(w)
$$
with $d\mu(w)$ being the Lebesgue measure of $\mathbb{R}^{n+1}$. Namely $\int_{\mathbb{S}^n} \log u_{z_0}$ is the weighted (and signed) volume of $\Omega^0_{z_0}\Delta B(1)$. In particular, for any $z_0$ with  $|\Omega^*_{z_0}|\le |B(1)|$, we have
$\int_{\mathbb{S}^n} \log u_{z_0}(x)\, d\theta(x)\ge 0$. Moreover, if $z_0\in \operatorname{Int}(\Omega)$ is such a point that $\mathcal{E}(\Omega)=\frac{1}{\omega_n}\int_{\mathbb{S}^n} \log u_{z_0}$, then
$$
\int_{\Omega^0_{z_0}} \frac{w}{|w|^{n+1}}\, d\mu(w)=0.
$$
Namely $z_0$ is the center of mass of $\Omega^*_{z_0}$ with respect to the weighted  measure $\frac{d\mu(w)}{|w|^{n+1}}$.
\end{proposition}
\begin{proof} The proof is via a similar calculation as the above:
\begin{eqnarray*}
\int_{\mathbb{S}^n} \log u_{z_0}(x)\, d\theta(x)&=&-\int_{\mathbb{S}^n}\int_1^{\frac{1}{u_{z_0}(x)}}\frac{1}{r}\, dr\, d\theta(x)\\
&=&\left( \int_{\{u_{z_0}(x)\ge1\}\subset \mathbb{S}^n} \int_{\frac{1}{u_{z_0}(x)}}^1  - \int_{\{u_{z_0}(x)<1\}\subset \mathbb{S}^n} \int^{\frac{1}{u_{z_0}(x)}}_1 \right)\frac{1}{|w|^{n+1}} \, d\mu(w)\\
&=& \left(\int_{B(1)\setminus\Omega^{0}_{z_0}}- \int_{\Omega^{0}_{z_0}\setminus B(1)}\right) \frac{1}{|w|^{n+1}}\, d\mu(w),
\end{eqnarray*}
hence the first part of the proposition. The nonnegativity of the entropy holds since on $\frac{1}{|w|^{n+1}}\ge 1$ on $B(1)\setminus \Omega^0_{z_0}$ and $\frac{1}{|w|^{n+1}}\le 1$ on $\Omega^0_{z_0}\setminus B(1)$.

The last claim can be proved via a similar calculation.
\end{proof}

The following lemma asserts that there exists a unique point $z_e\in \Omega$  such that the entropy $\mathcal{E}(\Omega)$ is attained. Such a point $z_e$ shall be called the {\it entropy point}.

\begin{lemma}\label{entropy-ex}  Given $\Omega$, a closed convex body there exists a unique $z_e\in \Omega$ such that
$\mathcal{E}(\Omega)=\frac1{\omega_n}\int_{\bfS^n} \log u_{z_e}(x)$.
\end{lemma}
 \begin{proof} The quantity $\frac 1{\omega_n} \int_{\bfS^n} \log u_{z_0}(x)$ is a function of $-z_0=(t_1, \cdots, t_{n+1})$ as
$$
F(t)=\frac 1{\omega_n} \int_{\bfS^n}\log \left(u(x)+\sum_{i=1}^{n+1} t_i x_i\right)d\theta(x).
$$
It is easy to see that the convexity of $\Omega$ implies that $u_{z_0}\ge 0$ for any $z_0\in \Omega$ and $F(t)$ is a strictly concave function of $t$. If for the sequence $\{p_n\}$ such that $\frac 1{\omega_n} \int_{\bfS^n} \log u_{p_n}(x) \to \mathcal{E}(\Omega)$, increasingly as $n \to \infty$. Without the loss of generality we may assume that $p_n \to p$. Then by Fatou's lemma, note that $\log u_{z}(x) \le \log \operatorname{diam}(\Omega)$ for any $z$ and $\log u_{p_n}(x)\to \log u_p(x)$, we have that
$$
\frac1{\omega_n}\int_{\bfS^n} -\log u_p(x) \le \frac 1{\omega_n}\liminf_{n \to \infty} \int -\log u_{p_n} =-\mathcal{E}(\Omega).
$$
On the other hand by the definition $\frac1{\omega_n}\int_{\bfS^n} \log u_p(x)\le \mathcal{E}(\Omega)$. Hence $\frac1{\omega_n}\int_{\bfS^n} \log u_p(x)= \mathcal{E}(\Omega)$. The uniqueness follows from the strictly concavity of $F(t)$ (as a function of $t\in \bfR^{n+1}$) and the convexity of $\Omega$.\end{proof}

We also denote $u_{z_e}$ by $u_e$. The next lemma strengthens the above result by asserting that in fact $z_e\in \operatorname{Int}(\Omega)$.

\begin{lemma}\label{e-p} If $\Omega$ is a bounded convex domain with  $\operatorname{Int}(\Omega)\ne \emptyset$, then $\mathcal{E}(\Omega)$ is attained by a unique support function $u_e>0$ such that
\begin{equation}\label{1stvar}\int_{\mathbb S^n}\frac{x_j}{u_e(x)}\, d\theta(x)=0.\end{equation} Moreover for any other support function $u\neq u_e$, $\mathcal{E}(\Omega)>\frac1{\omega_n}\int_{\mathbb S^n} \log u.$
\end{lemma}
\begin{proof} The main claim here is  that $u_e>0$ everywhere. Assuming this, the claimed (\ref{1stvar}) follows easily by the first variation. Namely express any support function as
\[u(x)=u_e(x)+\sum_{j=1}^{n+1}t_jx_j.\] By the maximum property of $u_e$, the first variation yields,
\[\int_{\mathbb S^n} \frac{x_j}{u_e(x)}\, d\theta(x)=0.\]

Suppose $u_e(x_0)=0$ for some $x_0\in \mathbb S^n$. Then by the convexity of $\Omega$ it is easy to see that $z_e$ must be on the boundary of $\Omega$. We may assume $z_e=0$, the origin. Now we {\it claim} that  there is a support hyperplane of $\Omega$ at the origin with
outer normal $\eta$ such that the line segment
\begin{eqnarray}\label{claim} L=\{-t\eta \quad | \quad 0<t<t_0\} \end{eqnarray}
is inside of $\Omega$, for some small $t_0$.

 We now prove this claim\footnote{We would like thank Gaoyong Zhang to communicating us the proof of {\it claim} \ref{claim}.}. First recall that for any $p\in \Omega$, the tangent cone $T^{\mathcal{C}}_p\Omega$ is defined as $\{\xi\, |\, \langle \xi, p-z_1\rangle \ge 0, \mbox{ for any } z_1 \mbox{ with }\operatorname{dist}(z_1, \Omega)=|z_1-p|\}$. The (out) normal cone $\mathcal{N}_p(\Omega)$ then is defined as $\{\eta\, |\, \langle \eta, \xi\rangle\le 0,\xi \in T^{\mathcal{C}}_p\Omega\}$. Now it is rather elementary to see that for any support hyperplane $H$ at $p$, which can be expressed  as the zero set of $f(z)=\langle \eta, z-p\rangle$,  with the property that  for all $z\in \Omega$, $f(z)\le 0$, $\eta\in \mathcal{N}_p(\Omega)$. Namely the outer normal of any support hyperplane must lies inside the normal cone. To prove the claim it suffices to show that $-\mathcal{N}_p(\Omega)$ intersects $\operatorname{Int}(\Omega)$, due to the convexity of $\Omega$. If  $-\mathcal{N}_p(\Omega)\cap \operatorname{Int}(\Omega)=\emptyset$, by the separation theorem (cf. Theorem 1.3.8 of \cite{Sch}), there must exists a hyperplane $H$ passing origin  which separates $\operatorname{Int}(\Omega)$ and $-\mathcal{N}_p(\Omega)$. This hyperplane must be a support hyperplane. But its out normal $\eta$  (with respect to $\Omega$) lies inside $ \mathcal{N}_p(\Omega)$. Hence it implies that $-\eta \in -\mathcal{N}_p(\Omega)$. This is a contradiction since $-\mathcal{N}_p(\Omega)$ is on the other (out) side of $H$ than the one  of $\Omega$. The {\it claim} (\ref{claim}) also follows from Theorem 1.12 of \cite{B}.

We may, without the loss of generality, assume that $\eta=(0,\cdots,0,1)$, the north pole is one with the property that the associated line segment $L$ defined by (\ref{claim}) lies inside $\operatorname{Int}(\Omega)$. Hence $\Omega$ is contained in the  half space $x_{n+1}\le 0$ and touches the hyperplane at the origin. For any $x=(x_1,\cdots, x_n,x_{n+1})\in \mathbb S^n$ with $x_{n+1}\ge 0$, let $N(x)=(x_1,\cdots, x_n,-x_{n+1})$ be its symmetric image with respect to $x_{n+1}=0$. By definition, $u_e(x)=\sup_{z\in \Omega}\langle z,x\rangle$. Since  $\Omega$ is closed, for each $x\in \mathbb S^n$, there is $z(x)\in  \Omega$ such that $u_e(x)=\langle z(x),x\rangle$.
Hence
\[u_{e}(N(x))\ge \langle z(x), N(x)\rangle \ge \langle z(x), x\rangle=u_{e}(x), \quad \forall x\in \mathbb S^n \quad \mbox{with $x_{n+1}\ge 0$},\]
here the fact $\langle z(x), \eta \rangle\le 0$ is used. Noticing that $z_e=0$ and $u_e(\eta)=0$ and obviously $u_e(N(\eta))>0$, the above inequality holds strict inequality for some $x\in \mathbb{S}^n$ consisting of a set of positive measure. Consider
new support function $u_s(x)=u_e(x)+s x_{n+1}$. By the fact that the line segment  $L$,  defined as (\ref{claim}),  lies in the interior of $\Omega$, $u_{s}(x)>0, \forall\, x\in \mathbb S^n, \forall\,  0<s<t_0$. On the other hand,
\begin{eqnarray*}\frac{d}{d s}\left.\left(\int_{\mathbb S^n}\log u_{s}\right)\right|_{s=0}&=&\int_{\mathbb S^n}\frac{x_{n+1}}{u_e(x)}\\
&=& \int_{\{x_{n+1}>0\}}\frac{x_{n+1}}{u_e(x)}+\int_{\{x_{n+1}<0\}}\frac{x_{n+1}}{u_e(x)}\\
&=& \int_{\{x_{n+1}>0\}}\left(\frac{x_{n+1}}{u_e(x)}-\frac{x_{n+1}}{u_e(N(x))}\right)>0,\end{eqnarray*}which is a contradiction to the definition of $u_e$. Therefore, $u_e(x)>0, \forall x\in \mathbb S^n$. \end{proof}

\medskip

In the rest of this section we derive some geometric estimates in terms of  the entropy. Let $\rho_{+}(\Omega) $  ($\rho_{-}(\Omega)$) be the outer (inner) radius of  a convex body $\Omega$. By definition, the outer radius is the  radius of the smallest ball which contains $\Omega$ and the inner radius is the radius of the biggest ball which is enclosed by $\Omega$. There is also a width function $w(x)$ which is defined as $u_{z_0}(x)+u_{z_0}(-x)$, where $u_{z_0}$ is the support function with respect to $z_0$.
 It is clear that $w(x)$ is independent of the choice of $z_0$. The $w_{+}$ and $w_{-}$ denote the maximum and minimum of $w(x)$. The following estimates  have been  known \cite{Andrews-thesis}
 \begin{equation}\label{andrews-eq1}
 \rho_{+}\le \frac{w_{+}}{\sqrt{2}}, \quad \quad \rho_{-}\ge \frac{w_{-}}{n+2}.\end{equation}
 Below we prove  a result relating these geometric quantities with the entropy.

\begin{corollary}\label{rho-e} For a convex body $\Omega$,
\begin{equation}\label{C0-upper} \max\{w_{+}, \rho_{+}(\Omega)\}\le C_n e^{\mathcal{E}(\Omega)},\end{equation}
where $C_n$ is a dimensional constant. There is also the lower estimate:
\begin{equation}\label{C0-lower} \min\{\rho_{-}(\Omega),  w_{-}\} \ge C'_n V(\Omega) e^{-n\mathcal{E}(\Omega)},\end{equation}
where $C'_n$ is another dimensional constant.
\end{corollary}
\begin{proof}
 The upper estimate can  be reduced to the corresponding upper  estimate of $w_{+}$ in view of (\ref{andrews-eq1}). Assume that $w(x_0)=w_{+}$. Without the loss of the generality we may assume that $u_{z_0}(x_0)\ge u_{z_0}(-x_0)$,  $z_0=0$. Hence $w_+\le 2u_0(x_0)$. Assume that $u_0(x_0)=\langle z_1, x_0\rangle$ for $z_1 \in \partial \Omega$. Applying the rotation we may also assume that $z_1=(0, \cdots, 0, a\rangle$, with $a=|z_1|$. Then $w_+\le 2a$. By the convexity, the line segment $tz_1$ (with $0\le t \le 1$) lies inside $\Omega$. It is also clear that the support function for this segment with respect to $\frac{z_1}{2}$ is $\frac{1}{2}|\langle z_1, x\rangle|$. Hence it  is bounded from above by $u_{\frac{z_1}{2}}(x)$.  Therefore
\begin{eqnarray*}
 \omega_n\, \log a -\omega_n\,  \log 2 +\int_{\bfS^n} \log |x_{n+1}|\, d\theta(x)&=& \int_{\bfS^n} \log \frac1 2|\langle z_1, x\rangle|\, d\theta(x) \\
&\le&\int_{\bfS^n} \log u_{\frac{z_1}{2}} \, d\theta(x)\\
&\le& \omega_n \mathcal{E}(\Omega).
\end{eqnarray*}
Notice that the integral on the left hand side  depends only on $n$.  This gives an
upper bound of $a$, hence an estimate for $w_{+}$. The lower bound on $\rho_-$ can be derived
out of this and the observation that $\Omega$ can be  enclosed in a cylinder with the  base of a ball of
radius $\rho_+$, and the height of $2w_-$. Hence
$$
n\omega_{n-1} \rho_+^{n}\cdot 2w_- \ge V(\Omega).
$$
The lower bound of $\rho_-$ follows from estimate of $\rho_{-}$ in terms of $w_{-}$.
\end{proof}

\section{ Gauss curvature flow and entropies}

First we recall the relation between the embedding $X:M \to \bfR^{n+1}$ of $M$, a closed convex hypersurface in $\bfR^{n+1}$ and   the related support function $u(x): \bfS^{n}\to \bfR$ of the enclosed convex body $\Omega$ (here we assume that $0\in \Omega$ and $u(x)=u_0(x)$):
$$
u(x)=\langle z, X(\nu^{-1}(z))\rangle
$$
where $\nu(y): M \to\bfS^n$ is the Gauss map. For convenience we also denote  $X(
\nu^{-1}(x))$ by $X(x)$ (namely $X(x)$, for $x\in \bfS^n$, denotes the embedding reparametrized via the Gauss map). The following equations are well-known \cite{Andrews-thesis}:
\begin{eqnarray}
X(x)&=& u(x) \cdot x +\bar{\nabla} u \label{1st}\\
\left(W^{-1}\right)^i_j&=& \bar{g}^{ik}\left(\bar{\nabla}_k \bar{\nabla}_j u +u \bar{g}_{kj}\right). \label{2nd}
\end{eqnarray}
Here $W=d\nu$ is the Weingarten map, $\bar{\nabla}$ is the covariant derivative of $\bfS^n$ with respect to standard induced metric $\bar{g}$ as the boundary of the unit ball in $\bfR^{n+1}$. It is clear from (\ref{1st}) that changing of the reference point $z_0$ in the support function amounts to translating by $-z_0$ on the embedding $X(x)$, and from (\ref{2nd}) that the Weingarten map $W$ is independent of the choice of the reference point $z_0$. Let $K(x)=\det( W)$ be the Gauss curvature. First we derive the following estimate on Chow's entropy  \cite{Chow} in terms of the entropy defined in the last section.

\begin{proposition}\label{two-entropy} Let $\Omega$ be a convex body with smooth boundary $M=\partial \Omega$ and volume $V(\Omega)=V(B(1))$. Let $K$ be the Gauss curvature of $M$. Then
\begin{equation}\label{chow-lb}
\mathcal{E}_C (\Omega) \doteqdot \frac{1}{\omega_n}\int_{M}K \log K\, d\sigma \ge  \mathcal{E}(\Omega)\ge 0.
\end{equation}
Here $d\sigma$ is the induced surface measure on $M$.
 Moreover $\mathcal{E}_C(\Omega)=\mathcal{E}(\Omega)$ if and only if $K=u_e$, and $\mathcal{E}_C(\Omega)=0$ if and only if $\Omega=B(1)$, the unit ball. For general $\Omega$,
 $$
 \mathcal{E}_C(\Omega)\ge \mathcal{E}(\Omega)-\log\left( \frac{V(\Omega)}{V(B(1))}\right).
 $$
\end{proposition}
\begin{proof} First observe that
$
\int_{M}K \log K \, d\sigma =\int_{\mathbb S^n} \log K\, d\theta.
$
On the other hand, recall
\begin{eqnarray*}
\frac{1}{\omega_n}\int_{\mathbb S^n} \frac{u}{K}\, d\theta &=& \frac{1}{\omega_n}\int_M \langle X, \nu\rangle \, d\sigma\\
&=& \frac{n+1}{\omega_n} V(\Omega).
\end{eqnarray*}
Hence the estimate via Jensen's inequality gives, in the case $V(\Omega)=V(B(1))$,
\begin{eqnarray*}
1&=& \frac{1}{\omega_n}\int_{\mathbb S^n} \frac{u}{K}\, d\theta \\
&=&\frac{1}{\omega_n} \int \exp\left(\log \left(\frac{u}{K}\right)\right)\, d\theta \\
&\ge& \exp\left(\frac{1}{\omega_n}\int_{\mathbb S^n} \log \left(\frac{u}{K}\right)\, d\theta\right).
\end{eqnarray*}
This implies that
$$
\frac{1}{\omega_n}\int_{\mathbb S^n}\log K \, d\theta \ge\frac{1}{\omega_n}\int_{\mathbb S^n} \log u\, d\theta.
$$
Since this estimate holds for support functions with respect to any $z_0\in \Omega$, we have the claimed estimate.
The equality case follows from Proposition \ref{key}.
\end{proof}
\begin{remark}
 An alternate argument below, using that $x-1-\log x\ge 0$, proves a similar result with a weaker estimate:
\begin{eqnarray*}
\frac{1}{\omega_n}\left((n+1)V(\Omega)-\omega_n\right)&=& \frac{1}{\omega_n}\int_{\mathbb S^n} \left(\frac{u}{K}-1\right)\, d\theta(x)\\
&\ge& \frac{1}{\omega_n}\int_{\mathbb S^n} \log \frac{u}{K}\, d\theta(x).
\end{eqnarray*}
Hence $\mathcal{E}_C(\Omega)-\mathcal{E}(\Omega) \ge -\frac{V(\Omega)}{V(B(1))}+1$.
\end{remark}

\begin{corollary} Let $\Omega$ and $M$ be as in Proposition \ref{two-entropy}. Let $\sigma_k(W)=\sum_{i_1<i_2\cdots <i_k} \lambda_{i_1}\cdots \lambda_{i_k}$ be the $k$-th elementary symmetric function of  (strictly speaking, eigenvalues $\{\lambda_i\}$ of) the Weingarten map. Then
\begin{eqnarray}
\frac{1}{\omega_n}\int_{\mathbb S^n} \frac{k! (n-k)!}{n!}\sigma_{k}(W)\, d\theta\,\ge 1; \quad \quad
\frac{1}{\omega_n}\int_{\mathbb S^n}\frac{k! (n-k)!}{n!} K\sigma_k(W)\, d\theta \ge 1. \label{chow-PI-rhs}
\end{eqnarray}
The equality holds in any inequality if and only if $\Omega=B(1)$.
\end{corollary}

The Gauss curvature flow
\begin{equation}\label{gauss1}
\frac{\partial   X(x, t)}{\partial t}=-K(x, t)\nu,
\end{equation}
which deforms the hypersurface $M_t$ along its inner normal with the speed given by it Gaussian curvature $K$,
has been studied since Firey's \cite{Firey} article. In terms of the support function, the flow can be expressed as
\begin{equation}\label{gauss2}
\frac{\partial  u(x, t)}{\partial t}=-\frac{1}{\det \left(\bar{g}^{ik}\left(\bar{\nabla}_k \bar{\nabla}_j u +u \bar{g}_{kj}\right)\right)}.
\end{equation}
Since the convexity of $M_t$ is preserved along the flow (\ref{gauss1}), the equation (\ref{gauss2}) in terms of the support function $u$ always makes sense.
In \cite{Tso} the existence of (\ref{gauss1}) has been proved and it was also shown that the flow will contract a convex hypersurface to a limiting point $z_\infty$. The main concern here is to understand what is the limiting shape of the evolving hypersurfaces $M_t$. To understand the asymptotic behavior of the flow we also consider the normalized flow:
\begin{equation}\label{gauss-nor}
\frac{\partial  u(x, t)}{\partial t}=u(x, t)-\frac{1}{\det \left(\bar{g}^{ik}\left(\bar{\nabla}_k \bar{\nabla}_j u +u \bar{g}_{kj}\right)\right)}
\end{equation}
which preserves the enclosed volume $V(\Omega_t)$, provided that the initial $V(\Omega_0)=V(B(1))$.
By suitable scaling (multiplying a factor $e^{\tau}$ to the support function $u$) and re-parametrization ($\tau=-\frac{1}{n+1}\log\left(\frac{T-t}{T}\right)$, with $T$ being the terminating time, which equals to $\frac{1}{n+1}$ under the above normalization, and relabeling $\tau$ by $t$ afterwards), the support function with respect to $z_\infty$ yields a long time {\it positive} solution to (\ref{gauss-nor}). Hence the study of the limiting shape is equivalent to the asymptotic of (\ref{gauss-nor}). When $\Omega$ is centrally symmetric it was shown by Firey that the solution of (\ref{gauss-nor}) converges to a round sphere. In dimension $n=2$, B. Andrews \cite{An3} proved the same result for any convex surfaces in $\mathbb{R}^3$.

In the later discussion we also denote $\bar{g}^{ik}\left(\bar{\nabla}_k \bar{\nabla}_j u +u \bar{g}_{kj}\right)$ by $A$, or $A_u$ to make clear the dependence,  and $-\frac{1}{\det(A)}$ by $\Psi$, viewing as a function of the tensor $A$.  Such a function  $\Psi$ satisfies $-n$-concavity property. Namely
\begin{equation}\label{n-concave}
\ddot \Psi (X, X)\le \frac{n+1}{n}\frac{(\dot \Psi(X))^2}{\Psi}.
\end{equation}
 When we discuss a solution to (\ref{gauss-nor}) we always assumes that $A>0$. The elliptic operator $\mathcal{L}\doteqdot (\dot \Psi_{A})_{ij}\bar{\nabla}_i \bar{\nabla}_j$, in terms of  a normal coordinate of $\mathbb{S}^n$, appears in the linearization  of (\ref{gauss-nor}):
  $$\frac{\partial}{\partial t}-\mathcal{L}-KH -1.$$
   If $u_1$ and $u_2$ are two convex (being the support function of a convex body) solutions to (\ref{gauss-nor}) with $u_1(x, 0)=u_2(x, 0)$, then $v=u_1-u_2$ satisfies, under the normal coordinates,
$$
\frac{\partial}{\partial t}v=\left(\int_0^1 (\dot\Psi (A_s))_{ij} ds \right) \bar{\nabla}_i\bar{\nabla}_j v +\left(\int_0^1 \dot\Psi (A_s)(\delta_{ij})ds\right) v +v
$$
with $A_s=\bar{\nabla}_i \bar{\nabla}_j u_s+u_s \delta_{ij}$ and $u_s=su_1+(1-s)u_2$. Hence $u_1(x, t)\equiv u_2(x, t)$. The following evolution equations are also well-known \cite{An3}.

\begin{proposition} Under the normalized flow (\ref{gauss-nor}), the following hold:
\begin{eqnarray}
\left(\frac{\partial}{\partial t} -\mathcal{L}\right) u &=& (n+1) \Psi +u -u\Psi H, \label{sup-pde}\\
\left(\frac{\partial}{\partial t} -\mathcal{L}\right) \Psi &=& -\Psi^2 H-n\Psi, \label{speed-pde}\\
\left(\frac{\partial}{\partial t} -\mathcal{L}\right) P &=& P-\Psi H P+\ddot \Psi_A(Q, Q). \label{acc-pde}
\end{eqnarray}
Here $H$ is the mean curvature of $M_t\doteqdot \partial \Omega_t$, $P=\frac{\partial \Psi}{\partial t}$, the time derivative of the speed, namely the acceleration, $Q=A_t$.
\end{proposition}
Noticing that $-\Psi H=\dot\Psi_A(\operatorname{id})$, the above two equations can be written as
\begin{eqnarray}
\left(\frac{\partial}{\partial t} -\mathcal{L}\right) u &=& (n+1) \Psi +u +u\dot\Psi_A(\operatorname{id}), \label{sup-pde1}\\
\left(\frac{\partial}{\partial t} -\mathcal{L}\right) \Psi &=& \Psi \dot\Psi_A(\operatorname{id})-n\Psi,
\label{speed-pde1}\\
\left(\frac{\partial}{\partial t} -\mathcal{L}\right) P &=& P +\dot\Psi_A (\operatorname{id})P+\ddot \Psi_A(Q, Q). \label{acc-pde1}
\end{eqnarray}
From these equations it is easy to see that (\ref{gauss-nor}) preserves the volume of the enclosed body. Precisely,
\begin{eqnarray*}
\Sigma(t)\doteqdot\int_{\bfS^n} \frac{u}{-\Psi} \, d\theta(x)&=& \int_{M_t} \langle X(y, t), \nu(y)\rangle d\sigma(y) \\
&=& \int_{\Omega_t} \operatorname{div} (X)\, d\mu_{y}\\
&=& (n+1) V(\Omega_t).
\end{eqnarray*}
A direct calculation using (\ref{sup-pde}), (\ref{speed-pde}) and divergence structure of the operator $\frac{\mathcal{L}}{\Psi^2}$,  yields
$$
\Sigma'(t)=(n+1) \left(\Sigma(t)-\omega_n\right).
$$
Since $\Sigma(0)-\omega_n =0$, this implies that $\Sigma(t)\equiv \omega_n$ for all $t$.

 The evolution equation on $A_{ij}\doteqdot u_{ij}+u\delta_{ij} $, namely the inverse of the Weingarten map $W^{-1}$,  under the normal coordinates is useful.

 \begin{proposition}\label{matrix} Under the normal coordinates, for solution to (\ref{gauss-nor}) the tensor $A_{ij}$ satisfies
\begin{equation}\label{eveq-A1}
\left(\frac{\partial }{\partial t}-\mathcal{L}\right)
A_{ij}=
-KH A_{ij}+  A_{ij}+(n-1)K
\bar{g}_{ij}+\ddot \Psi_A(\bar{\nabla}_i  A, \bar{\nabla}_j
 A).
\end{equation}
Here $\Psi=-K$, $H$ is the mean curvature, namely the sum of the eigenvalues of $A^{-1}$.
 \end{proposition}
As before (\ref{eveq-A1}) can be written as
\begin{equation}\label{eveq-A2}
\left(\frac{\partial }{\partial t}-\mathcal{L}\right)
A_{ij}=
-\dot\Psi_A(\operatorname{id}) A_{ij}+  A_{ij}-(n-1)\Psi
\bar{g}_{ij}+\ddot \Psi_A(\bar{\nabla}_i  A, \bar{\nabla}_j
 A).
\end{equation}
 Below we show the derivation of corresponding equation on $A_{ij}$ when  $u$ is  instead a solution of (\ref{gauss2}) since the corresponding equation readily yields an upper estimate for the Hessian of $u$, for the un-normalized solution $u$. By the equation (\ref{gauss2}) we have  that
$\frac{\partial}{\partial t}A_{ij}=\bar{\nabla}_i\bar{\nabla}_j \Psi
+\Psi \bar{g}_{ij}$. Now we compute
\begin{eqnarray*}
\bar{\nabla}_j \Psi &=& \dot \Psi_A (\bar{\nabla}_j A),\\
\bar{\nabla}_i \bar{\nabla}_j \Psi &=& \dot \Psi_A (\bar{\nabla}_i
\bar{\nabla}_j A)+\ddot \Psi_A(\bar{\nabla}_iA, \bar{\nabla}_j A),\\
\bar{\nabla}_i\bar{\nabla}_j A_{kl}&=&\bar{\nabla}_i\bar{\nabla}_j
\bar{\nabla}_k \bar{\nabla}_l u +\bar{\nabla}_i\bar{\nabla}_j u
\bar{g}_{kl}.
\end{eqnarray*}
The commutator formulae yield
\begin{eqnarray*}
\bar{\nabla}_j \bar{\nabla}_k \bar{\nabla}_l u &=& \bar{\nabla}_k
\bar{\nabla}_j \bar{\nabla}_l u -\bar{R}_{lpkj}\bar{\nabla}_p u,\\
\bar{\nabla}_i\bar{\nabla}_j \bar{\nabla}_k \bar{\nabla}_l u &=&
\bar{\nabla}_i\left(\bar{\nabla}_k \bar{\nabla}_l \bar{\nabla}_j
u-\bar{R}_{lpkj}\bar{\nabla}_p u\right),\\
&=& \bar{\nabla}_k\bar{\nabla}_i \bar{\nabla}_l \bar{\nabla}_j u
-\bar{R}_{lpkj}\bar{\nabla}_i \bar{\nabla}_p u
-\bar{R}_{jpki}\bar{\nabla}_p \bar{\nabla}_l u
-\bar{R}_{lpki}\bar{\nabla}_j\bar{\nabla}_p u,\\
\bar{\nabla}_k\bar{\nabla}_i \bar{\nabla}_l \bar{\nabla}_j u
&=&\bar{\nabla}_k\bar{\nabla}_l \bar{\nabla}_i \bar{\nabla}_j
u-\bar{R}_{jpli} \bar{\nabla}_p\bar{\nabla}_k u.
\end{eqnarray*}
Here $\bar{R}_{ijkl}=\delta_{ik}\delta_{jl}-\delta_{il}\delta_{jk}$ is the curvature tensor of $\bfS^n$.
Putting together we have that
\begin{eqnarray*}
\bar{\nabla}_i\bar{\nabla}_j \bar{\nabla}_k \bar{\nabla}_l u
&=&\bar{\nabla}_k\bar{\nabla}_l \bar{\nabla}_i \bar{\nabla}_j
u-\bar{R}_{jpli} \bar{\nabla}_p\bar{\nabla}_k
u-\bar{R}_{lpkj}\bar{\nabla}_i \bar{\nabla}_p u\\
&\quad & -\bar{R}_{jpki}\bar{\nabla}_p \bar{\nabla}_l u
-\bar{R}_{lpki}\bar{\nabla}_j\bar{\nabla}_p u.
\end{eqnarray*}
Now  using that $(\dot \Psi_A)^{kl}=K A^{kl}$, where $(A^{ij})$ is the inverse of $(A_{ij})$, we have that
\begin{eqnarray*}
\bar{\nabla}_i \bar{\nabla}_j \Psi &=&
KA^{kl}\left(\bar{\nabla}_i\bar{\nabla}_j \bar{\nabla}_k
\bar{\nabla}_l u \right)+KH \bar{\nabla}_i\bar{\nabla}_j u+\ddot \Psi_A (\bar{\nabla}_iA, \bar{\nabla}_j A)\\
&=&KA^{kl}\left(\bar{\nabla}_k\bar{\nabla}_l(A_{ij}-u \bar{g}_{ij})\right)-2KH\bar{\nabla}_i\bar{\nabla}_j\, u
+2KA^{kl}\left(\bar{\nabla}_k\bar{\nabla}_l u\right)\bar{g}_{ij}\\
&\quad& +KH \bar{\nabla}_i\bar{\nabla}_j u+\ddot
\Psi_A (\bar{\nabla}_iA, \bar{\nabla}_j A)\\
&=&KA^{kl}\left(\bar{\nabla}_k\bar{\nabla}_lA_{ij}\right)-KH\bar{\nabla}_i\bar{\nabla}_j\, u+
+KA^{kl}\left(\bar{\nabla}_k\bar{\nabla}_l
u\right)\bar{g}_{ij}+\ddot \Psi_A (\bar{\nabla}_iA, \bar{\nabla}_j
A)\\
&=&KA^{kl}\left(\bar{\nabla}_k\bar{\nabla}_lA_{ij}\right)-KHA_{ij}+nK
\bar{g}_{ij}+\ddot \Psi_A (\bar{\nabla}_iA, \bar{\nabla}_j A).
\end{eqnarray*}
Combining the above we arrive at the following parabolic
equation on $A_{ij}$:
\begin{equation}\label{eq-aij}
\frac{\partial }{\partial
t}A_{ij}=KA^{kl}\left(\bar{\nabla}_k\bar{\nabla}_lA_{ij}\right)-KHA_{ij}+(n-1)K
\bar{g}_{ij}+\ddot \Psi_A (\bar{\nabla}_iA, \bar{\nabla}_j A).
\end{equation}The equation (\ref{eveq-A1}) follows  similarly if $u$ satisfies (\ref{gauss-nor}) instead.
Let $B_{ij}=\bar{\nabla}_i\bar{\nabla}_j u$, the Hessian of $u$, then if $u$ is a solution to (\ref{gauss1}),  $B$ satisfies:
\begin{equation}\label{eq-bij}
\frac{\partial }{\partial
t}B_{ij}=KA^{kl}\left(\bar{\nabla}_k\bar{\nabla}_lB_{ij}\right)-KHB_{ij}+2nK
\bar{g}_{ij}-2uHK \bar{g}_{ij}+\ddot \Psi_A (\bar{\nabla}_iA,
\bar{\nabla}_j A).
\end{equation}
The immediate consequence of the above is a upper bound on $B_{ij}$.
Let
$$
B_S(t)=\max_{x\in \bfS^n}\max_{X\in T_x\bfS^n, |X|=1} X^i X^j
\bar{\nabla}_i\bar{\nabla}_j u.
$$
If $B_S(t_0)=\max_{t\in [0, T)}B_S(t)$,  using the concavity of $\ddot \Psi$, we have that at an  extremal
point $(x_0, t_0)$, where $B_S(t_0)$ is achieved, by the maximum principle,
$$
H B_S(t_0)\le 2n -2uH.
$$
Hence, via the Cauchy-Schwarz estimate $H\ge n K^{1/n}$,
$$
B_S(t_0)\le \frac{2}{K^{1/n}}.
$$
Using $\inf_{M_t} K \ge \inf_{M_0} K$ we have the uniform upper
bound
\begin{equation}\label{C2-upper}
\left(\bar{\nabla}_i\bar{\nabla}_j u\right)(x, t) \le
\frac{2}{\inf_{M_0} K^{\frac{1}{n}}}\bar{g}_{ij}(x,
t)+\max_x\left(\bar{\nabla}_i\bar{\nabla}_j u\right)(x, 0)
\end{equation}
which recovers a key $C^2$-estimate of \cite{Tso} in the proof of the existence and the convergence to a point for the un-normalized flow.

Making use of the computation above we also have the following evolution equation on $|X|^2=|\bar{\nabla} u|^2+u^2$.
\begin{equation}\label{pos-ev1}
\left(\frac{\partial }{\partial t}-\mathcal{L}\right)|X|^2 =2|X|^2 -2(\dot\Psi_A)_{ij} \bar{\nabla}_i\bar{\nabla}_k u \bar{\nabla}_j\bar{\nabla}_k u+2(n+1)u\Psi +2u^2(\dot\Psi_A)(\operatorname{id}).
\end{equation}

The following discussion  reveals the relation between the entropy $\mathcal{E}(\Omega)$ and the normalized Gauss curvature flow (\ref{gauss-nor}). First note that the equilibrium for (\ref{gauss-nor}) satisfies the equation
\begin{equation}\label{equi}
u(x, t)\cdot \det \left(\bar{g}^{ik}\left(\bar{\nabla}_k \bar{\nabla}_j u +u \bar{g}_{kj}\right)\right)=1.
\end{equation}
Such a solution is also called a shrinking soliton of the Gauss curvature flow.

We now consider the first variation of $\mathcal{E}(\Omega)$ under constraint $V(\Omega)=V(B(1))$.
Fix $\Omega$, by Lemma \ref{e-p}, there exists a unique $z_e\in \operatorname{Int}(\Omega)$ such that $\mathcal{E}(\Omega)=\frac{1}{\omega_n} \int_{\bfS^n} \log u_e(x)\, d\theta(x)$.  Moreover such a $u_e$ satisfies
\begin{equation}\label{e1}\int_{\mathbb S^n}\frac{x_j}{u_e}\, d\theta(x)=0, \quad \forall j=1,\cdots, n+1.\end{equation}
 Let $\Omega_\eta$ be a family of convex body such that $\Omega_0=\Omega$. In terms of support functions, we have a family of functions  $v_\eta\in C^{2}(\mathbb S^n)$
such that
\[A_\eta=((v_\eta)\delta_{ij}+(v_{\eta})_{ij})>0.\]
 We assume that in addition $v_\eta$ satisfies (\ref{e1}).  Hence  $\mathcal{E}(\Omega_\eta)=\frac{1}{\omega_n} \int_{\mathbb S^n}\log v_\eta\, d\theta(x)$.
 Write $v_\eta(x)=u_e(x)+\rho(\eta,x)$, where $\rho(0,x)=0,$ for all $x\in \mathbb S^n$. Below we abbreviate $v_\eta$ by $v$, $u_e$ by $u$. As before, the constraint
$V(\Omega_\eta)=V(B(1))$ implies
\begin{equation}\label{e2} \frac{1}{\omega_n}\int_{\mathbb S^n}v\det(A_v)=1.\end{equation}
Recall that we also have
\begin{equation}\label{e3} \int_{\mathbb S^n}\frac{x_j}{v}=0, \quad \forall j=1,\cdots, n+1; \forall \eta.\end{equation}
\begin{equation}\label{e4} \mathcal{E}(\Omega_\eta)=\frac{1}{\omega_n}\int_{\mathbb S^n}\log v.\end{equation}

\begin{proposition} If $u$, the unique support function which achieves the entropy, is a critical point of $\mathcal{E}(\Omega)$, viewed as a functional of $\Omega$,   under the constraint that $V(\Omega)=V(B(1))$,
it  must be a solution to (\ref{equi}), namely a shrinking soliton. Namely a critical point to $\mathcal{E}(\Omega)$ must be a shrinking soliton to the Gauss curvature flow. Moreover, the converse is also true.
\end{proposition}
\begin{proof}
Differentiate (\ref{e2}) and (\ref{e4}) in $\eta$ and then set $\eta=0$. Applying the Lagrangian multiplier method, at any critical point $u$,
there exists a  $\lambda\in \mathbb R$, such that (in view of (\ref{e2}))
\begin{eqnarray}\label{e5} \int_{\mathbb S^n} \rho^{'}\det(A_u)=\lambda\int_{\mathbb S^n} \frac{\rho^{'}}{u},  \quad \forall \rho^{'}, \mbox{with} \int_{\mathbb S^n} \frac{\rho^{'}x_j}{u^2}=0, \forall j=1,\cdots, n+1. \end{eqnarray}
Here we have used that the $\frac{\partial \det(A_u)}{\partial A_{ij}}\bar{\nabla}_i \bar{\nabla}_j$ is self-adjoint.
Let $\mathcal{N}_u=\operatorname{span}\{\frac{x_j}{u}, j=1,\cdots, n+1\}$. Note that
$$
\int_{\mathbb S^n} \det(A_u) x_j=\int_{\partial \Omega} \langle \nu, e_j\rangle =0.
$$
Since both $u(\det(A_u)-\frac{\lambda}{u})$ and
$\frac{\rho^{'}}{u}$ belong to $\mathcal{N}_u^\perp$ and  $\frac{\rho^{'}}{u}$ is arbitrary in $\mathcal{N}_u^\perp$ and $u>0$, we must have
\begin{equation}\label{e6}
\det(A_u)=\frac{\lambda}{u}.\end{equation}
As $V(\Omega)=V(B(1))$, we conclude that $\lambda=1$. To check the converse, from (\ref{e2}) we conclude that
$$
\int_{\mathbb S^n} \rho' \det(A_u)=0,
$$
which readily implies that $\int_{\mathbb S^n} \frac{\rho'}{u}=0$. Note that (\ref{e3}) holds automatically for $u$ with $u=K$. Namely for the soliton, the origin is the entropy point.
\end{proof}

The next result gives a lower estimate on the the volume of $\Omega^*_0$, the dual of $\Omega$ with respect to the origin when $\Omega$ (more precisely $u$,  the support function with respect to the origin) is a soliton of the Gauss curvature flow.

\begin{proposition}\label{dualvol-low} Assume $u$ is a soliton with associated body $\Omega$ (namely  $u=K$, with enclosed volume being the one of the unit ball). Then the following holds.

 \begin{itemize} \item[(i)] The origin is the entropy point of $\Omega$;

 \item[(ii)] The volume of $\Omega^*_0$ satisfies
\begin{equation}\label{lower-dual}
V(\Omega^*_0)\ge V(B(1)).
\end{equation}
\end{itemize}

In particular, if the origin is the Santal\'o point of $\Omega$ then $\Omega=B(1)$.
\end{proposition}
\begin{proof} Observe that for any $1\le j\le n+1$
$$
0=\int_M \langle \nu(z), e_j\rangle d\sigma= \int_{\mathbb{S}^n} \frac{x_j}{K}\, d\theta(x),
$$
which implies that
$$
\int_{\mathbb{S}^n} \frac{x_j}{u}\, d\theta(x) =0.
$$
This implies that the origin is the entropy point.

Let $X(x)=u(x) \, x+\bar{\nabla}u(x)$ be the position vector of $M_t$. Observe that for any support function $u$ of a convex body
\begin{eqnarray*}
\frac{1}{\omega_n}\int_{\mathbb S^n} \frac{u}{K |X|^{n+1}}\, d\theta(x)&=&\frac{1}{\omega_n}\int_{\partial \Omega}\frac{\langle X, \nu\rangle}{|X|^{n+1}}d\sigma \\
&=& \frac{1}{\omega_n}\int_{\partial B(\epsilon)} \frac{1}{\epsilon^{n}} d\sigma \\
&=& 1.
\end{eqnarray*}
Here we have used that $\operatorname{div}( \frac{X}{|X|^{n+1}})=0$.
The claimed lower estimate on the dual volume  follows as
\begin{eqnarray*}
\frac{V(\Omega^*_0)}{\omega_n}&=&\frac{1}{n+1}\aint_{\mathbb{S}^n} \frac{1}{u^{n+1}}\\
&\ge& \frac{1}{n+1}\aint_{\mathbb S^n} \frac{u}{K |X|^{n+1}}\\
&=&\frac{1}{n+1}.
\end{eqnarray*}
The last statement follows, since when  the origin is the Santal\'o point, $V(\Omega^*_0)\le V(B(1))$ by the Blaschke-Santal\'o inequality, hence  the equality holds in above estimates.  In particular, it implies  that $|X|=u$ and $\bar{\nabla}u=0$, namely $u$ is a constant.
\end{proof}

\begin{remark} One can also prove (\ref{lower-dual}) using the isoperimetric inequality:
$\aint_{\mathbb{S}^n} \frac{1}{K}\, d\theta(x)\ge 1$.
\end{remark}

Relating to Proposition \ref{two-entropy} and the normalized Gauss curvature flow (\ref{gauss-nor}),  Chow \cite{Chow} proved that $\mathcal{E}_C(\Omega_t)$ is monotone non-increasing along the flow. The following theorem is of fundamental importance to the later discussions.

\begin{theorem}\label{mono1} Along the flow (\ref{gauss-nor}) the entropy $\mathcal{E}(\Omega_t)$ is monotone non-increasing. Moreover for any $t_1\le t_0$
\begin{equation}\label{eq-entropy-de-speed}
\mathcal{E}(\Omega_{t_0})-\mathcal{E}(\Omega_{t_1})\le \int_{t_1}^{t_0}\left(\mathcal{E}(\Omega_t)-\mathcal{E}_C(\Omega_t)\right)\, dt \le 0.
\end{equation}
\end{theorem}
\begin{proof} At some point $t_0$ assume that $\mathcal{E}(\Omega_{t_0})=\frac{1}{\omega_n}\int_{\bfS^n} \log u_{e(t_0)}$, where $u_{e(t_0)}$ is the support function with respect to a unique entropy point $z_e(t_0) \in \operatorname{Int}(\Omega)$. Hence for $t<t_0$ but very close to $t_0$, one still has that
$u_{e(t)}(x, t)\doteqdot u(x, t)-\langle \exp{(t-t_0)}\, z_e(t_0), x\rangle>0$. If $u(x, t)$ is a solution to (\ref{gauss-nor}), so is $u_{e(t)}(x, t)$. Now calculate
\begin{eqnarray*}
\frac{d}{dt}  \aint_{\bfS^n} \log u_{e(t)}(x, t)&=&\aint_{\bfS^n} \frac{u_{e(t)}-K}{u_{e(t)}}\\
&=& 1-\aint_{\bfS^n} \frac{K}{u_{e(t)}}\\
&=&- \aint_{\bfS^n}\left(\sqrt{\frac{K}{u_{e(t)}}}-\sqrt{\frac{u_{e(t)}}{K}}\right)^2\\
 &\le& 0.
\end{eqnarray*}
This implies that there exists $\delta>0$ and for $t\in (t_0-\delta,  t_0)$, $$\mathcal{E}(\Omega_t)\ge \aint_{\bfS^n} \log u_{e(t)}(x, t)\ge \aint_{\bfS^n} \log u_{e(t_0)}(x,t_0)=\mathcal{E}(\Omega_{t_0}),$$
which proves the first claim. Making use of the above calculation again we have that
$$
\mathcal{E}(\Omega_{t_0})-\mathcal{E}(\Omega_{t_1})\le \int_{t_1}^{t_0} \aint_{\mathbb{S}^n}\left( 1-\frac{K}{u_{e(t)}}\right)\, d\theta\, dt.
$$
Using  $1-x\le -\log x$ and some elementary estimates, established (\ref{eq-entropy-de-speed}) for $t_1\in (t_0-\delta, t_0)$. The continuity argument can be applied to conclude the same for all $t_1\le t_0$.
\end{proof}
The proof above is a modification of that of Firey \cite{Firey}, in which he introduced the entropy $\mathcal{E}_F(\Omega_t)=\aint_{\mathbb{S}^n} \log u(x, t)$ and showed that it is monotone non-increasing along the flow. Now we have that $\mathcal{E}_C(\Omega_t)\ge \mathcal{E}(\Omega_t)\ge \mathcal{E}_F(\Omega_t)$.

\section{$C^0$--Estimates}

Let $u(x, t)$ be a long time solution to (\ref{gauss-nor}). By translation we may assume that $z_\infty=0$.
Combining Corollary \ref{rho-e} and Theorem \ref{mono1} we have  an upper bound of $\rho_+$, hence an upper bound of $u(x, t)$. Since the volume is preserved along the normalized Gauss curvature flow, by John's lemma, it follows $\rho_{-}$ is bounded from below. The estimate of upper bound of $u(x,t)$ was proved by Hamilton first in \cite{Hamilton-gauss} using a different argument.\footnote{We were informed recently by Xujia Wang that he also obtained some similar results in his unpublished manuscript.}

The main result of this section is to establish a uniform lower bound of $u(x, t)$. Since we assume that $z_\infty $, the limit point which lies inside all $\Omega_t$ evolving by (\ref{gauss2}), is the origin, we have a solution $u(x, t)$ to (\ref{gauss-nor}) with $u(x, t)>0$ for all $(x, t)\in \mathbb{S}^n\times[0, \infty)$.

We start with a similar lower bound for the support function with respect to the Santal\'o point, which motivates the $C^0$-estimates. This is based on the following gradient estimate on a support function $u$ of a convex body:
\begin{equation}\label{gradient}
\max_{\mathbb S^n}|\bar{\nabla} u|\le \max_{\mathbb S^n} u.
\end{equation}
This gradient estimate can be proved by the following observation.
Due to the positivity of $\bar{\nabla}_i\bar{\nabla}_j u+ u\delta_{ij}$, one can conclude that at the maximum point of $|\bar{\nabla} u|^2+u^2$, $\bar{\nabla}u=0$. Hence
$\max_{\mathbb S^n}|\bar{\nabla} u|\le \max_{\mathbb S^n} u$. Geometrically this is clear since $X=\bar{\nabla}u+u\, x$ is the position vector with length square $|X|^2=|\bar{\nabla}u|^2+u^2$, which attains its maximum for some $X_0$ parallel to $x$.

As $\rho_{-}$ is bound from below, if one is willing to shift the origin, a lower bound of the support function would follow. We task is to bound the support function from below without shifting for all $t$. In this regard, the entropy point plays important role. To motivate the discussion, we first consider the Santal\'o point.
\begin{proposition}\label{santalo-le} If $u_s$ is the support function with respect to the Santal\'o point of $\Omega$, then
\begin{equation} u_s(x)\ge c(n) V(\Omega)e^{-n\mathcal{E}(\Omega)},\end{equation}
where $c(n)>0$ is a dimensional constant.
\end{proposition}
\begin{proof}   By the Blaschke-Santal\'o inequality,
\begin{eqnarray*} V(\Omega_s^*)=\frac1{n+1}\int_{\mathbb S^n} \frac1{u^{n+1}_s}\le \frac{V^2(B(1))}{V(\Omega)}.\end{eqnarray*}
Let $m=u_s(x_0)$ be the minimum value of
$u_s$ (attained at some $x_0$).  By (\ref{gradient}),
$\max_{\mathbb S^n}|\bar{\nabla} u_s|\le \max_{\mathbb S^n} u_s\le 2\rho_{+}$. Therefore, in a geodesic ball $\bar{B}_{x_0}(r)$ (inside $\mathbb{S}^n$) with $r= \frac m{\rho_{+}}$, we have
$u_{s}(x)\le 2m$. In turn,
\begin{eqnarray*} \frac{V^2(B(1))}{V(\Omega)}\ge \frac1{n+1}\int_{\mathbb S^n} \frac1{u^{n+1}_s}\ge \tilde C_n m^{-(n+1)}r^n=\tilde C_n \frac{\rho_{+}^{-n}}m.\end{eqnarray*}  The result now follows from  Corollary \ref{rho-e}.\end{proof}

Next is the main result of this section, which is based on establishing a similar result for $u_{e(t)}$ where $e(t)$ is the entropy point of the convex body $\Omega_t$.

\begin{theorem}\label{C0}
Suppose $u(x, t)>0$ is the solution of (\ref{gauss-nor}) with initial data $u(x, 0)=u_0(x)>0$, where $u_0(x)$ is the support function of $\Omega_0$ with $V(\Omega_0)=V(B(1))$ and $\mathcal{E}(\Omega_0)\le A $. Then there is a
positive $\epsilon =\epsilon(n,\Omega_0)>0$ such that
\begin{equation}\label{eq-c01} u(x,t)\ge \epsilon, \quad \forall t\ge 0, \forall x\in \mathbb S^n.\end{equation}
 \end{theorem}

The proof is built upon several lemmas. For each bounded closed convex body $\Omega$, we denote $e(\Omega)$  the unique entropy point of $\Omega$. For each $p\in \Omega$, recall that $u_p$ is the support function of $\Omega$ with respect to $p$.

\begin{lemma}\label{e-concave}  For each $\Omega$, there is $D>0$ depending only on $n$ and the diameter of $\Omega$ such that for any $p\in \Omega$,
\begin{equation}\label{E-est1} \frac{1}{\omega_n}\int_{\mathbb S^n} \log u_p \le \mathcal{E}(\Omega)-D \operatorname{dist}^2(p,e(\Omega)).\end{equation}
\end{lemma}
\begin{proof} Since $u_p$ is bounded from above by $2\rho_{+}$, $\frac1{u_p}$ is bounded from below. As in Lemma \ref{entropy-ex}, consider $F(t)=\frac{1}{\omega_n}\int_{\bfS^n} \log u_p=\frac{1}{\omega_n}\int_{\bfS^n} \log \left(u_e+\langle x, e-p\rangle \right)$ with $t=e-p$. The direct calculation shows that
$$
\frac{\partial^2\,  F(t)}{\partial t_i \partial t_j}=-\int_{\bfS^n} \frac{x_i x_j}{\left(u_e+\langle x, t\rangle\right)^2}\, d\theta(x).
$$
By Taylor's theorem, if write $t=|t| \, a$ with $a=\frac{e-p}{|e-p|}$ we have that
$$F(t)\le F(0)- C |t|^2 \int_{\bfS^n} \langle a, x\rangle^2\, d\theta(x).$$
Here $C$ is a constant only depending on the upper bound of $\rho_{+}$.
Now (\ref{E-est1}) follows from the fact that the integral on the right hand side is a constant depending only on $n$. \end{proof}

Note that  by Corollary \ref{rho-e}, there exists an upper bound of $\rho_+$ depending only on $A$, the upper bound of the entropy.
\medskip

For each $A>0, B>0$, consider the collection of bounded closed convex sets
\begin{equation} \Gamma^A_B=\{ \Omega \subset \mathbb R^{n+1} |\quad \mbox{$\Omega$ is a closed convex subset, $0\in \Omega$}, V(\Omega)\ge B, \mathcal{E}(\Omega)\le A.\} \end{equation}

\begin{lemma}\label{e-cont} Suppose $\Omega_k\in \Gamma^A_B$ is a sequence of convex bodies with the property that $0\in  \Omega_k, \forall k$.
Suppose $\lim_{k\to \infty} \Omega_k=\Omega_0$, then
\[ \lim_{k\to \infty} \mathcal{E}(\Omega_k)=\mathcal{E}(\Omega_0).\]
Moreover, there is $\delta(A,B,n)>0$ depending only on $n,A,B$
such that the entropy point $e_{\Omega}$ satisfies the following estimate:
\begin{equation}\label{e-est} dist(e_{\Omega}, \partial \Omega)\ge \delta(A,B,n), \quad \forall \Omega\in \Gamma^A_B.\end{equation} \end{lemma}

 \begin{proof} By the Lemma \ref{rho-e}, $\forall \Omega\in \Gamma^A_B$,  $\rho_+(\Omega)\le C(n,A)$ for some $C(n,A)>0$.
Since the volume is bounded from below we also have  $\forall \Omega\in \Gamma^A_B$,  $\rho_-(\Omega)\ge c(n,A,B)$  for some
$c(n,A,B)>0$.  By Lemma \ref{e-p}, the entropy point of $\Omega_0$, $e_{\Omega_0}\in \Omega_0$. Therefore, when $k$ large,  $e_{\Omega_0}\in \Omega_k$.
Again by Lemma \ref{e-p},
\begin{equation}\label{lemm34-1} \mathcal{E}(\Omega_0)=\frac{1}{\omega_n} \int_{\mathbb S^n} \log u^{\Omega_0}_{e(\Omega_0)}=\lim_{k\to \infty} \frac{1}{\omega_n} \int_{\mathbb S^n} \log u^{\Omega_k}_{e(\Omega_0)}\le \lim_{k\to \infty}  \mathcal{E}(\Omega_k).\end{equation}
Here $u^{\Omega_k}_p$ denote the support function of $\Omega_k$ with respect to $p$.

On the other hand, since $u^{\Omega_k}_p \le 2\rho_+(\Omega_k)\le 2C(n,A)$ for each $p\in   \Omega_k$, $\log u^{\Omega_k}_p$ is bounded from above. As $\Omega_k\in \Gamma^A_B$, by estimate (\ref{est-basic11}) of Proposition \ref{key}
\begin{eqnarray*}
\frac{1}{\omega_n}\int_{\mathbb S^n}\log \left(\frac{u^{\Omega_k}_{e(\Omega_k)}}{2C(n,A)}\right)&\ge&  \mathcal{E}(\Omega_k)-\log (2C(n,A))\\
&\ge&  \frac{\log (\frac{B}{V(B(1))})}{n+1}-\log (2C(n,A)).\end{eqnarray*}
That is, \begin{equation}\label{u-k-c}\int_{\mathbb S^n}\left|\log \left(\frac{u^{\Omega_k}_{e(\Omega_k)}}{2C(n,A)}\right)\right|\le C, \forall k.\end{equation}
Let $p=\lim_{k\to \infty} e(\Omega_k)$. Noticing that $\log (\frac{u^{\Omega_k}_{e(\Omega_k)}}{2C(n,A,B)})\le 0$,
by Fatou's Lemma,
\[\int_{\mathbb S^n}\log \left(\frac{u^{\Omega_0}_{p}}{2C(n,A,B)}\right) \ge \lim \sup_{k\to \infty} \int_{\mathbb S^n}\log \left(\frac{u^{\Omega_k}_{e(\Omega_k)}}{2C(n,A)}\right).\]
This yields
\begin{equation}\label{lemm34-2} \mathcal{E}(\Omega_0)\ge \limsup_{k\to \infty}  \mathcal{E}(\Omega_k).\end{equation}
Combining (\ref{lemm34-1}) and (\ref{lemm34-2}) proves the first claim of the lemma.

For the second part, suppose the statement (\ref{e-est}) is not true. there is a sequence
$\{\Omega_k\in \Gamma^A_B\}$ such that
\[\operatorname{dist}(e_{\Omega_k}, \partial \Omega_k) \to 0, \quad k\to \infty.\]
By Blaschke selection theorem (cf. \cite{Sch}, Theorem 1.8.6), there exists a subsequence of $\{\Omega_k\in \Gamma^A_B\}$, which we still denote as
$\Omega_k$, converges to a convex body $\Omega_0$. Let $p=\lim_{k\to \infty} e(\Omega_k)$. By the assumption $\operatorname{dist}(e_{\Omega_k}, \partial \Omega_k) \to 0$, we have $p\in
\partial \Omega_0$.  The support function $u_p$ of $\Omega_0$ vanishes at $p$. By the first part of the lemma, $\mathcal{E}(\Omega_0)=\lim_{k\to \infty} \mathcal{E}(\Omega_k)$. Hence $\Omega_0\in \Gamma^A_B$. Again, argue as before using Fatou's Lemma,
\[\mathcal{E}(\Omega_0)=\lim_{k\to \infty} \mathcal{E}(\Omega_k)=\lim_{k\to \infty} \frac{1}{\omega_n} \int_{\mathbb S^n} \log u^{\Omega_k}_{e(\Omega_k)} \le \frac{1}{\omega_n} \int_{\mathbb S^n} \log u_p.\]
This is a contradiction to Lemma \ref{e-p}. \end{proof}

 Now consider the positive solution to (\ref{gauss-nor}).  We first observe an easy consequence of the uniqueness.
\begin{proposition}\label{uniqueness-pos} For any given convex body $\Omega$ with normalized volume,
there at most one positive solution of (\ref{gauss-nor}) which exists on $\mathbb{S}^n \times[0, \infty)$ such that $u(x, 0)$ is a support function of $\Omega$.
\end{proposition}
\begin{proof}
  Suppose $v$ is another positive solution, at $t=0$, $v(x,0)=u(x,0)-\sum_{i=1}^{n+1} a_ix_i$.
It is easy to check $\tilde v(x,t)=u(x,t)-e^t\sum_{i=1}^{n+1} a_ix_i$ is a solution of the normalized Gauss curvature flow, namely satisfies $\tilde v_t=-K+ \tilde v.$ Here note that $A_{\tilde{v}}=A_u$, hence $\Psi(A_u)=\Psi(A_{\tilde{v}})$.
Therefore, $\tilde v=v$. Hence if $a \neq 0$, $v$ can not be bounded! Therefore there exists only one positive solution to (\ref{gauss-nor}) on $\mathbb S^n \times[0, +\infty)$. \end{proof}

For each $\Omega_t$ corresponding to $u(x,t)$, let $\mathcal{E}(t)\doteqdot\mathcal{E}(\Omega_t)$. We know
$\mathcal{E}(t)\ge 0$ and monotonically decreasing. Let $\mathcal{E}_{\infty}\doteqdot\lim_{t\to \infty} \mathcal{E}(t)$.

\begin{lemma}\label{cov-e} Let  $u(x, t)$ be the unique positive solution of (\ref{gauss-nor}).  Then
\begin{equation}\label{100e}  \aint_{\mathbb S^n} \log u(x,t)\ge \mathcal{E}_{\infty} +\int_t^\infty \aint_{\mathbb{S}^n}\left(\sqrt{\frac{K}{u}}-\sqrt{\frac{u}{K}}\right)^2,\,  \forall t\ge 0.\end{equation}
In particular, $\mathcal{E}(t)\ge \mathcal{E}_F(t)\ge \mathcal{E}_\infty$. \end{lemma}

 \begin{proof} For each $T_0>$ fixed, pick $T > T_0$. Let $a^{T}=(a^{T}_1,\cdots, a^{T}_{n+1})$ be
the entropy point of $\Omega_T$. Set $u^T=u- e^{t-T}\sum_{i=1}^{n+1} a^T_ix_i$, it can be checked that
\begin{equation}\label{TGCF} u^T_t=-K+u^T.\end{equation}
Since both the origin and the entropy point $a^T$ are in $\operatorname{Int}(\Omega_T)$,
\[|a^T|\le 2\rho^+(t)\le C.\] If $T$ large enough, $u^T(x, 0)>0, \forall x\in \mathbb S^n$. We also
know $u^T(x, T)>0, \forall x\in \mathbb S^n$ since the entropy point is an interior point of $\Omega_T$.
If $u^T(x_0, t_0)\le 0$ for some $0<t_0<T, x_0\in \mathbb S^n$, the equation (\ref{TGCF}) implies $u^T(x_0, t)<0$ for
all $t>t_0$, which contradicts to $u^T(x, T)>0$. Hence  $u^T(x, t)>0, 0\le t\le T, x\in \mathbb S^n$. By equation (\ref{TGCF}), a similar calculation as in Theorem \ref{mono1} shows
\begin{eqnarray}\label{Te1} \frac{d}{dt}\left(\int_{\mathbb S^n} \log u^T(x,t)\right)=-\int_{\mathbb S^n} \left(\sqrt{\frac{K(x,t)}{u^T(x,t)}}-\sqrt{\frac{u^T(x,t)}{K(x,t)}}\right)^2, \forall\,  0\le t\le T.\end{eqnarray}
Hence
\begin{eqnarray*} \frac1{\omega_n} \int_{\mathbb S^n} \log u^T(x,0)-\mathcal{E}(T)= \frac1{\omega_n}\int_{t=0}^{T}\int_{\mathbb S^n} \left(\sqrt{\frac{K(x,t)}{u^T(x,t)}}-\sqrt{\frac{u^T(x,t)}{K(x,t)}}\right)^2 .\end{eqnarray*}
Since $T_0<T$,
\begin{eqnarray*}\frac1{\omega_n} \int_{\mathbb S^n} \log u^T(x,0)-\mathcal{E}(T)\ge \frac1{\omega_n} \int_{t=0}^{T_0}\int_{\mathbb S^n} \left(\sqrt{\frac{K(x,t)}{u^T(x,t)}}-\sqrt{\frac{u^T(x,t)}{K(x,t)}}\right)^2.\end{eqnarray*}
Now let $T\to \infty$, as $u^T(x,t)\to u(x,t)$ uniformly for $0\le t\le T_0, x\in \mathbb S^n$, we obtain
\begin{eqnarray}\label{Te2}\frac1{\omega_n} \int_{\mathbb S^n} \log u(x,0)-\mathcal{E}_{\infty}\ge \frac1{\omega_n} \int_{t=0}^{T_0}\int_{\mathbb S^n} \left(\sqrt{\frac{K(x,t)}{u(x,t)}}-\sqrt{\frac{u(x,t)}{K(x,t)}}\right)^2.\end{eqnarray}
Now (\ref{100e}), for $t=0$, follows directly from (\ref{Te2}) since $T_0$ is arbitrary. In the above if we replace $0$ by any $t\le T$ we obtain (\ref{100e}). The proof of the lemma is complete. \end{proof}

Lemma \ref{cov-e} has the following immediate consequence.
\begin{corollary}
$$
\lim_{t\to \infty} \mathcal{E}_C (\Omega_t)=\lim_{t\to\infty}\mathcal{E}(\Omega_t)=\mathcal{E}_\infty.
$$\end{corollary}
\begin{proof} Since $\mathcal{E}_C(\Omega_t)\ge \mathcal{E}(\Omega_t)$,  $\lim_{t\to \infty} \mathcal{E}_C (\Omega_t)\ge \mathcal{E}_\infty$.
Assume that the equality does not holds. Then there exists $\delta>0$, and  for sufficiently  large $t_0$ we have that
$ \mathcal{E}_C(\Omega_t)-\mathcal{E}(\Omega_t)\ge \delta$ for $t\ge t_0$. This contradicts to (\ref{100e}) since the integral on the right hand side is finite, and bounds $\int_t^\infty  \mathcal{E}_C(\Omega_s)-\mathcal{E}(\Omega_s)\, ds$ from the above. This is a contradiction, which proves the claim.
\end{proof}

 Now we are ready to prove Theorem \ref{C0}.

\begin{proof} (of Theorem \ref{C0})  Since $\mathcal{E}(\Omega_t)\to \mathcal{E}_{\infty}$. By (\ref{100e}),
\[\mathcal{E}_{\infty}\le \frac1{\omega_n} \int_{\mathbb S^n} \log u(x,t)\le \mathcal{E}(\Omega_t).\]
That is,
\[ 0\le \mathcal{E}(\Omega_t)-\frac1{\omega_n} \int_{\mathbb S^n} \log u(x,t ) \to 0, \mbox{ as } t\to \infty.\]
Note $u$ is the support function of $\Omega_t$ with respect to the origin, by Lemma \ref{e-concave},
$e(\Omega_t) \to 0$ as $t\to \infty$. The claimed lower estimate now  follows from (\ref{e-est}) in Lemma \ref{e-cont}. \end{proof}

The proof effectively shows that there exists $C=C(\Omega_0, n)$ such that if $e(t)=e(\Omega_t)$ the entropy point of $\Omega_t$,
\begin{equation}\label{entropy-to-0}
|e(t)|^2 \le C \left(\mathcal{E}(t)-\aint_{\mathbb{S}^n} \log u(x, t)\right).
\end{equation}
Finally the following corollary summarizes  Corollary \ref{rho-e}, Theorem \ref{mono1} and Theorem \ref{C0}.

\begin{corollary}\label{C0-sum} Let $u(x, t)$ be as in Theorem \ref{C0}. Then there exists $\Lambda=\Lambda(\Omega_0, n)>0$ such that
\begin{equation}\label{C0-2sides}
\frac{1}{\Lambda}\le u(x, t)\le \Lambda.
\end{equation}
\end{corollary}

\section{$C^2$-estimates and the convergence}

 In this section we derive uniform $C^2$-estimates out of the $C^0$-estimate (\ref{C0-2sides}).
The first is a upper estimate, which was first proved by Hamilton \cite{Hamilton-gauss}. We provide a  different proof here using the $C^0$-estimate.

\begin{theorem}\label{C2-u} Suppose $u(x, t)\ge a>0$ is the solution of (\ref{gauss-nor}) with initial data $u(x, 0)=u_0(x)$, where $u_0(x)>0$ is the support function of $\Omega_0$ with $V(\Omega_0)=V(B(1))$.
There exists  a constant $C=C(a, n)>0$ such that
\begin{equation}\label{upper-c2}
K(x, t)\le C.
\end{equation}
\end{theorem}
\begin{proof}  Consider the quantity $Q\doteqdot \frac{K}{2u-a}$. Applying the evolution equations (\ref{sup-pde}) and (\ref{speed-pde})
\begin{eqnarray*}
\left(\frac{\partial}{\partial t} -\mathcal{L}\right)Q&=& \frac{K^2 H -nK}{2u-a}-2K\frac{-(n+1)K+u+uKH}{(2u-a)^2} +2\dot\Psi_{ij}\bar{\nabla}_i Q \bar{\nabla}_j \log (2u-a)\\
&=& \frac{-a K^2 H +2(n+1)K^2 -(2u-a)nK -2uK}{(2u-a)^2}  +2\dot\Psi_{ij}\bar{\nabla}_i Q \bar{\nabla}_j \log (2u-a).
\end{eqnarray*}
Now apply the maximum principle, if $m(t)=\max_{x\in \bfS^n} Q(x, t)$ is achieved at $(x_0, t)$, then at that point we have that
\begin{eqnarray*}
0&\le& \frac{-a K^2 H +2(n+1)K^2 -(2u-a)nK -2uK}{(2u-a)^2}\\
&\le& m(t)^2 \left(-aH +2(n+1)\right).
\end{eqnarray*}
Noting that $K\le \left(\frac{H}{n}\right)^n$, we then deduce that at $(x_0, t)$,
$$
K\le \left(\frac{2(n+1)}{n\cdot a}\right)^n
$$
which in turn implies that
$$
m(t)\le \left(\frac{2(n+1)}{n}\right)^n\frac{1}{a^{n+1}}.
$$
The claimed estimate now follows from the above.
\end{proof}

We remark that in \cite{Hamilton-gauss}, Hamilton obtained the above estimate (cf. Corollary on page 156 of \cite{Hamilton-gauss}) by using the sharp differential estimate of Chow (which is also referred as a differential Harnack estimate, as well as a Li-Yau-Hamilton type estimate) and the entropy formula of Chow \cite{Chow}. Hamilton's estimate is built upon a lower estimate of  $\frac{u(x, t)}{K(x, t)}$. Our proof of Theorem \ref{C2-u} avoids the use of Chow's entropy formula and his  differential  estimate \cite{Chow}, but based on the $C^0$-lower bound.  Below we include a slightly stronger result on the lower estimate of $\frac{u(x, t)}{K(x, t)}$.

\begin{proposition}\label{hamilton1} Let $u$ be a solution to the un-normalized flow (\ref{gauss2}) with the reference point being the limit point, when $t\to T$. Then
\begin{equation}\label{hamilton-im}
\frac{u(x, t)}{K(x, t)}\ge (n+1)t^{\frac{n}{n+1}} \left(T^{\frac{1}{n+1}}-t^{\frac{1}{n+1}}\right).
\end{equation}
\end{proposition}
Since $T\ge t$, the above (\ref{hamilton-im}) implies $\frac{u(x, t)}{K(x, t)}\ge \left(\frac{t}{T}\right)^{\frac{n}{n+1}}(T-t)$, a result of Hamilton \cite{Hamilton-gauss}.

\begin{proof} By the differential estimate of Chow \cite{Chow}[Theorem 3.7],   we deduce  that, with respect the parametrization via the Gauss map,
$$
-\Psi_t -\frac{n}{(n+1)t} \Psi \ge 0.
$$
Then the direct calculation shows that $y(t)=\frac{u}{-\Psi}$ satisfies the estimate:
$$
y'(t)\le -1+\frac{n}{(n+1)t} y(t).
$$
Noticing that $y(t)\to 0$ as $t\to T$, integrating the above from $t$ to $T$ yields
$$
-t^{-\frac{n}{n+1}}y(t) \le -(n+1)\left(T^{\frac{1}{n+1}}-t^{\frac{1}{n+1}}\right).
$$
Hence we have the claimed estimate.
\end{proof}

Note that for the solution $u(x, t)$ to the normalized flow (\ref{gauss-nor}), the estimate (\ref{hamilton-im}) implies that
\begin{equation}\label{hamilton-im-nor}
\frac{u(x, t)}{K(x,t)}\ge \frac{1}{n+1}\left(1-e^{-(n-1)t}\right)^{n/n+1}
\end{equation}
which together with Corollary \ref{rho-e}, Theorem \ref{mono1} gives another proof of Theorem \ref{C2-u}.

For the $C^2$-estimate we first need the following lower bound on the Gauss curvature $K(x, t)$.

\begin{theorem}\label{thm-lower-k} Suppose $u(x, t)>0$ is a positive solution of (\ref{gauss-nor}), obtained from the un-normalized flow (\ref{gauss2}), with initial data $u(x,0)=u_0(x)$, where $u_0(x)>0$ is the support function of $\Omega_0$ with $V(\Omega_0)=V(B(1))$. Then there exists a constant $\epsilon_1=\epsilon(n, \Omega_0)>0$ such that
\begin{equation}\label{lower-C2}
K(x, t)\ge \epsilon_1.
\end{equation}
\end{theorem}

\begin{proof} For this estimate, it is more convenient to work with the un-normalized flow (\ref{gauss2}). Let $T$ be the terminating time (which is $\frac{1}{n+1}$ by our normalization). Then the claimed estimate is equivalent to
\begin{equation}\label{equ1}
K(x, t) (T-t)^{\frac{n}{n+1}}\ge \epsilon_1.
\end{equation}
For the proof we recall Theorem 3.7 of \cite{Chow} under the Gauss map parametrization:
\begin{equation}\label{eq-chow}
K(x, t) t^{\frac{n}{n+1}}\le K(x, t') t'^{\frac{n}{n+1}}
\end{equation}
for any $0<t\le t'<T$. Since it is sufficient to prove (\ref{equ1}) for $t\ge \frac{T}{2}$, the estimate (\ref{eq-chow}) implies that
\begin{equation}\label{rel-mono}
K(x, t)\le 2^{\frac{n}{n+1}} K(x, t').
\end{equation}
The two sided $C^0$-estimate (\ref{C0-2sides}) implies that for the un-normalized support function $u(x, t)$ it satisfies that
\begin{equation}\label{equ-2sides}
\frac{1}{\Lambda} (T-t)^{\frac{1}{n+1}}\le u(x, t)\le \Lambda (T-t)^{\frac{1}{n+1}}.
\end{equation}
Let $$
\alpha =\left(\frac{1}{2\Lambda^2}\right)^{n+1}, \quad h_j=\frac{T}{2}\alpha^j, \mbox{ and } \quad t_j=T-h_j \mbox{ for } j=0, 1, \cdots.
$$
Clearly $t_j \to T$ as $j\to \infty$. The above estimate (\ref{equ-2sides}) implies that
\begin{eqnarray}
u(x, t_j)-u(x, t_{j+1}) &\ge& \frac{1}{\Lambda} h_j^{\frac{1}{n+1}}-\Lambda h_{j+1}^{\frac{1}{n+1}} \nonumber\\
&=&\frac{1}{\Lambda} \left(\frac{T}{2}\right)^{\frac{1}{n+1}}\alpha^{\frac{j}{n+1}}-\Lambda \left(\frac{T}{2}\right)^{\frac{1}{n+1}}\alpha^{\frac{j+1}{n+1}}\nonumber\\
&=& \frac{1}{2\Lambda} h_j^{\frac{1}{n+1}}. \label{pot-low}
\end{eqnarray}
The Gauss curvature flow equation implies that for any $t'<T$
$$
u(x, t')=\int_{t'}^T K(x, t)\, dt
$$which in turn implies that
\begin{equation}\label{ffcal}
u(x, t_j)-u(x, t_{j+1}) =\int_{t_j}^{t_{j+1}} K(x, t)\, dt.
\end{equation}
Now we claim that there exists $s_j\in [t_j, t_{j+1}]$ such that
\begin{equation}\label{low-seq}
K(x, s_j) (T-s_j)^{\frac{n}{n+1}}\ge \frac{1}{4(n+1)\Lambda}.
\end{equation}
Otherwise we would have that
\begin{eqnarray*}
\int_{t_j}^{t_{j+1}} K(x, t)\, dt &\le& \frac{1}{4(n+1)\Lambda} \int_{t_j}^{t_{j+1}} (T-t)^{-\frac{n}{n+1}}\, dt\\
&=&\frac{1}{4(n+1)\Lambda}\int_{h_{j+1}}^{h_j} \tau^{-\frac{n}{n+1}}\, d\tau\\
&\le& \frac{1}{4\Lambda} h_j^{\frac{1}{n+1}}.
\end{eqnarray*}
A contradiction to (\ref{pot-low}) and (\ref{ffcal}) !

Now the claimed estimate (\ref{equ1}) can be derived from (\ref{low-seq}) and (\ref{rel-mono}). First we claim that
\begin{equation}\label{low-seq2}
K(x, t_{j+1})(T-t_{j+1})^{\frac{n}{n+1}}\ge \frac{1}{4(n+1)\Lambda} \left(\frac{\alpha}{2}\right)^{\frac{n}{n+1}}.
\end{equation}
This can be proven via the estimates
\begin{eqnarray*}
K(x, t_{j+1})(T-t_{j+1})^{\frac{n}{n+1}}&\ge& \frac{1}{2^{\frac{n}{n+1}}}K(x, s_j)h_{j+1}^{\frac{n}{n+1}}\\
&=& \frac{1}{2^{\frac{n}{n+1}}} K(x, s_j) \alpha ^{\frac{n}{n+1}}h_j^{\frac{n}{n+1}}\\
&\ge&\left(\frac{\alpha}{2}\right)^{\frac{n}{n+1}}K(x, s_j)(T-s_j)^{\frac{n}{n+1}}
\end{eqnarray*}
and (\ref{low-seq}). The claimed estimate (\ref{equ1}) follows by another iteration of the above argument applying (\ref{low-seq2}) instead. Namely for $t\in [t_j, t_{j+1}]$, we have that
\begin{eqnarray*}
K(x, t) (T-t)^{\frac{n}{n+1}}&\ge& \frac{1}{2^{\frac{n}{n+1}}}K(x, t_j) (T-t)^{\frac{n}{n+1}}\\
&\ge&\frac{1}{2^{\frac{n}{n+1}}}K(x, t_j) h_{j+1}^{\frac{n}{n+1}}\\
&\ge& \left(\frac{\alpha}{2}\right)^{\frac{n}{n+1}} K(x, t_j)(T-t_j)^{\frac{n}{n+1}}.
\end{eqnarray*}
Hence we conclude that for any $t\in [t_1, T]$,
$$
K(x, t) (T-t)^{\frac{n}{n+1}}\ge \left(\frac{\alpha}{2}\right)^{\frac{2n}{n+1}}\frac{1}{4(n+1)\Lambda}.
$$
The claimed result follows from the above easily.
\end{proof}

The next result provides an upper bound on $\sigma_1(A)$, the trace of $A_{ij}=\bar{\nabla}_i \bar{\nabla}_j u+u\delta_{ij}$. Noting that $\det\left(\bar{\nabla}_i \bar{\nabla}_j u+u\delta_{ij}\right)=K^{-1}$, together they provide an upper estimate of $|\bar{\nabla}_i \bar{\nabla}_j u+u\delta_{ij}|$, hence $|\bar{\nabla}_i \bar{\nabla}_j u|$,  the Hessian of $u$.

\begin{theorem}\label{c2-sharp}
 Suppose $u(x, t)>0$ is the solution of (\ref{gauss-nor}) with initial data $u(0,x)=u_0(x)$, where $u_0(x)>0$ is the support function of $\Omega_0$ with $V(\Omega_0)=V(B(1))$.
There exists  a constant $C>0$, depending on $n, \Omega_0$ such that
\begin{equation}\label{upperc2}
\operatorname{trace}\left(\bar{\nabla}_i \bar{\nabla}_j u+u\delta_{ij}\right) \le C.
\end{equation}
Moreover the symmetric tensor $A$ has the lower estimate:
\begin{equation}\label{positivity}
\bar{\nabla}_i \bar{\nabla}_j u+u\bar{g}_{ij}\ge \frac{1}{C} \bar{g}_{ij}.
\end{equation}
\end{theorem}

\begin{proof} We denote by $\sigma_i(A)$ (or simply $\sigma_i$)  the $i$-the symmetric function of the symmetric tensor $A_{ij}=\bar{\nabla}_i \bar{\nabla}_j u+u\delta_{ij}$. The previous result implies that $\sigma_n \ge \frac{1}{C_1}$, where $C_1$ is the positive constant from Theorem \ref{C2-u}. We  recall Newton's inequality (namely the function $\log \frac{\sigma_k}{C^k_n}$, with $C^k_n$ being the binomial coefficient,  is a concave function of $k$):
\begin{equation}\label{newton}
\frac{\sigma_{n-1}}{n}\ge \left(\frac{\sigma_1}{n}\right)^{\frac{1}{n-1}}\sigma_n ^{\frac{n-2}{n-1}}.
\end{equation}
The concavity of $\ddot \Psi$ together with (\ref{eveq-A1}) in Proposition \ref{matrix} implies that
\begin{equation}\label{trace1}
\left(\frac{\partial}{\partial t} -\mathcal{L}\right) \sigma_1 \le -\frac{\sigma_1 \sigma_{n-1}}{\sigma_n^2}+\sigma_1 +\frac{n(n-1)}{\sigma_n}-\frac{n+1}{n}\frac{|\bar{\nabla}K|^2}{K}.
\end{equation}
Let $m(t)\doteqdot \max_{x\in \bfS^n}\sigma_1(x, t)$. Then at $(x_0, t)$, where $m(t)$ is achieved we have that
\begin{eqnarray*}
0&\le& -\frac{\sigma_1 \sigma_{n-1}}{\sigma_n^2}+\sigma_1 +\frac{n(n-1)}{\sigma_n}\\
&\le& -n^{\frac{n-2}{n-1}}\frac{\sigma_1^{\frac{n}{n-1}}}{\sigma_n^{\frac{n}{n-1}}}+\sigma_1+n(n-1)C_1\\
&\le& -C_2\sigma_1^{\frac{n}{n-1}}+\sigma_1+C_1'.
\end{eqnarray*}
Here in the second last inequality we applied (\ref{newton}) and the upper estimate of $K(x,t)$, and in the last inequality we applied the lower estimate of $K(x, t)$ established in Theorem \ref{C2-u}. The claimed result (\ref{upperc2}) follows from the application of the maximum principle to the above estimate. The estimate (\ref{positivity}) follows from
Theorem \ref{C2-u} and (\ref{upperc2}).
\end{proof}

Combining Corollary \ref{rho-e}, Theorem \ref{mono1}, Theorem  \ref{C0}, Theorem \ref{C2-u} and Theorem \ref{c2-sharp}, as well as the gradient estimate (\ref{gradient}), we conclude that there exists a positive constant $C$ depending only on the initial data such that for the unique positive solution to (\ref{gauss-nor})
\begin{equation}\label{C2-norm}
\|u(\cdot, t)\|_{C^2(\mathbb S^n)} \le C.
\end{equation}

Since (\ref{gauss-nor}) is a concave parabolic equation, by   Krylov's theorem \cite{kry} and the standard theory on the parabolic  equations, estimates (\ref{C2-norm}) and (\ref{positivity})  imply the bounds on all derivatives (space and time) of $u(x, t)$. More precisely, for any $k\ge 3$, there exists $C_k\ge 0$, depending only on the initial value such that for $t\ge 1$
\begin{equation}\label{Ck}
\|u(\cdot, t)\|_{C^k(\mathbb S^n)}\le C_k.
\end{equation}
Now for any $T>0$ and sequence $\{t_j\}\to \infty$, consider $u_j(x, t)\doteqdot u(x, t-t_j)$. We have the following result on the sequential convergence.

\begin{proposition}\label{s-con} After passing to a subsequence, on $\bfS^n \times [-T, T]$, $\{u_j\}$ converges in the $C^\infty$-topology to a smooth function $u_\infty(x)$ which is a  self-similar solution to (\ref{equi}).
\end{proposition}
\begin{proof} By the proof of Theorem \ref{C0}, we have that for $t\in [-T, T]$,
$$
\lim_{j\to \infty} \frac{1}{\omega_n} \int_{\bfS^n} \log u_j(x, t)\, d\theta(x) \to \mathcal{E}_\infty.
$$
Hence $u_\infty(x, t)$ satisfies
$$
\frac{1}{\omega_n} \int_{\bfS^n} \log u_\infty(x, t)\, d\theta(x) =\mathcal{E}_\infty.
$$
$u_\infty$ is also a solution to (\ref{gauss-nor}) and positive by Theorem \ref{C0}. Hence by the proof of Theorem \ref{mono1} we conclude that
$$
\frac{u_\infty(x, t)}{K(x, t)}=\frac{K(x, t)}{u_\infty(x, t)}
$$
which implies that $(u_\infty)_{t}(x, t)=0$. Hence we have the claimed result.
\end{proof}

\section{Uniform convergence and the stability of the solitons}

Combining Theorem 2 of \cite{Andrews-IMRN} with Proposition \ref{s-con} we have the following result.

\begin{theorem}\label{c-infty}
The normalized GCF (\ref{gauss-nor}) converges in $C^\infty$-topology to a smooth soliton $u_\infty$ ($M_\infty$) which satisfies that $K(x)>0$ and the soliton equation:
$$
u\det (u\operatorname{id} +\bar{\nabla}^2 u)=1.
$$
\end{theorem}

It remains an interesting question to see if the round sphere (ball) is the unique compact soliton. For this sake we
 consider the following functional for $u>0$ with $A_u$ being positive definite
$$
\mathcal{J}_1(u)\doteqdot \aint_{\mathbb{S}^n} \log u -\frac{1}{n+1}\log \left(\aint_{\mathbb{S}^n} u\det(A_u)\right)+\frac{1}{2}\left(\aint_{\mathbb{S}^n}u\det(A_u)-1\right)^2.
$$
Here $\aint_{\mathbb{S}^n}=\frac{1}{\omega_n}\int_{\mathbb{S}^n}$.
If $v=u+\eta \rho$ is a variation, then
\begin{eqnarray*}
\left.\frac{d}{d\eta} \mathcal{J}_1(v)\right|_{\eta=0}&=&\aint_{\mathbb{S}^n} \frac{\rho}{u}-\frac{\aint_{\mathbb{S}^n}\rho \det(A_u)}{\aint_{\mathbb{S}^n}u\det(A_u)}+(n+1)\left(\aint_{\mathbb{S}^n}u\det(A_u)-1\right)\aint_{\mathbb{S}^n}\rho \det(A_u).
\end{eqnarray*}
Here we have used that
$$
\int u \sigma_n^{ij}(A) (A_\rho)_{ij} =\int \rho \sigma_n^{ij}(A) (A_u)_{ij}=n\int \rho \det(A_u)
$$
with $\sigma_n^{ij}(A)$ denotes the cofactor of $A_{ij}$ in $\det(A)$, which can also be expressed as $KW^{ij}$ with $(W^{ij})$ being the Weingarten map.
Hence the Euler-Lagrange equation of $\mathcal{J}_1(u)$ is
\begin{equation}\label{el-J1}
0=\frac{1}{u}-\frac{\det(A_u)}{\aint_{\mathbb{S}^n}u\det(A_u)}+(n+1)\left(\aint_{\mathbb{S}^n}u\det(A_u)-1\right)\det(A_u).
\end{equation}
Multiplying $u$ on the both sides of (\ref{el-J1}) and integrate on $\mathbb{S}^n$ we have that
$$
\aint_{\mathbb{S}^n} (u\det(A_u)-1)\, dx=0.
$$
This together with (\ref{el-J1}) implies that $u=\frac{1}{\det(A_u)}$. Hence we have the following proposition.

\begin{proposition}\label{soliton=el-J1} The critical point of functional $\mathcal{J}_1(u)$ among all positive smooth functions $u$ with $A_u>0$  satisfies the soliton equation $u=K$.
\end{proposition}

Similarly we can compute the second variation of the functional $\mathcal{J}_1$:
\begin{eqnarray*}
\left.\frac{d^2}{d\eta^2} \mathcal{J}_1(v_\eta)\right|_{\eta=0}&=& -\aint_{\mathbb{S}^n} \frac{\eta^2}{u^2}-\frac{\aint \eta \sigma_n^{ij}(\eta_{ij}+\eta \delta_{ij})}{\aint_{\mathbb{S}^n} u \det(A_u)}+(n+1)\left(\frac{\aint_{\mathbb{S}^n} \eta \det(A_u)}{\aint_{\mathbb{S}^n} u\det(A_u)}\right)^2\\
&\quad& (n+1)^2\left(\aint_{\mathbb{S}^n} \eta \det(A_u)\right)^2.
\end{eqnarray*}
Hence if $u\equiv 1$, making use that it is a critical point with $\aint u\det(A_u)=1$ we deduce that
\begin{eqnarray*}
\left.\frac{d^2}{d\eta^2} \mathcal{J}_1(v_\eta)\right|_{\eta=0}&=&-\aint_{\mathbb{S}^n} \eta^2 -\aint_{\mathbb{S}^n} \eta (\bar{\Delta} \eta +n \eta)+(n+1)(n+2)\left(\aint_{\mathbb{S}^n}\eta\right)^2\\
&=&\aint_{\mathbb{S}^n} |\bar{\nabla}\eta|^2 -(n+1)\eta^2 +(n+1)(n+2)\left(\aint_{\mathbb{S}^n}\eta\right)^2.
\end{eqnarray*}
This computation, together with the spectra of the sphere,  proves the following stability result.
\begin{proposition}\label{stable} The unit sphere/ball, namely the soliton with
$u\equiv 1$,  is stable among the variations $v_\eta=u+\eta$ with $\eta \perp \operatorname{span} \{ 1, x_1, \cdots, x_{n+1}\}$.
\end{proposition}

\bigskip

\noindent {\it Acknowledgements.} The first author would like to thank Xiuxiong Chen for useful discussions in 2001. Both authors would like to thank Ben Andrews, Ben Chow, Toti Daskalopoulos,  Richard Hamilton and Deane Yang for their interests.

\bibliographystyle{amsalpha}

\end{document}